\documentclass [12pt,oneside,a4paper,mathscr,reqno]{amsart}

\usepackage{amscd,amsmath,amssymb,euscript}

\usepackage{geometry} 
\usepackage{amsfonts}
\usepackage{amsmath,amssymb,amsthm,amsxtra}
\usepackage{float}
\usepackage[colorlinks,
  linkcolor=RoyalBlue, anchorcolor=Periwinkle,
  citecolor=Orange, urlcolor=Emerald]{hyperref}
\usepackage[usenames,dvipsnames]{xcolor}
\usepackage{enumitem}
\usepackage{graphicx}
\usepackage{subfigure}
\usepackage{bookmark}
\usepackage{tikz}\usetikzlibrary{matrix}
\usepackage{url}

\usepackage[all]{xy}

\usetikzlibrary{decorations.pathreplacing}
\usetikzlibrary{matrix,arrows}

\def\graphc{Emerald}
\def\vertexc{blue}
\def\halfc{blue}
\def\thirdc{blue}

\date{}

\title{Stability conditions and the $A_2$ quiver}
\author{Tom Bridgeland, Yu Qiu and Tom Sutherland}

\newtheorem{thm}{Theorem}[section]

\newtheorem{prop}[thm]{Proposition}
\newtheorem{lemma}[thm]{Lemma}

\theoremstyle{definition}
\newtheorem{defn}[thm]{Definition}

\newtheorem{thm*}[thm]{Theorem$^*$}
\newtheorem{remark}[thm]{Remark}

\newtheorem*{example*}{Example}

\renewcommand{\leq}{\leqslant}
\renewcommand{\geq}{\geqslant}

\newcommand{\mat}[4]{\begin{pmatrix}#1&#2\\#3&#4\end{pmatrix}}

\newcommand {\A}{\mathcal A}

\newcommand{\D}{\operatorname{D}}
\newcommand{\T}{\mathcal{T}}
\newcommand{\h}{\mathfrak{h}}
\newcommand{\Br}{\operatorname{Br}}
\newcommand {\Hom}{\operatorname{Hom}}
\newcommand {\Aut}{\operatorname{Aut}}
\newcommand {\Auts}{\operatorname{Aut}_*}
\newcommand{\Sph}{\operatorname{Sph}_*}
\newcommand {\Ext}{\operatorname{Ext}}
\newcommand {\Rep}{\operatorname{Rep}}
\renewcommand {\H}{\operatorname{\mathcal H}}
\newcommand{\Tw}{\operatorname{Tw}}

\newcommand{\Stab}{\operatorname{Stab}}
\renewcommand{\Im}{\operatorname{Im}}
\renewcommand{\Re}{\operatorname{Re}}
\renewcommand{\P}{\mathbb{P}}
\newcommand{\reg}{\operatorname{reg}}
\newcommand{\PSL}{\operatorname{PSL}}
\newcommand{\cS}{\mathcal{S}}
\newcommand{\R}{\mathbb R}
\newcommand{\FF}{\mathcal{F}}

\newcommand{\lra}{\longrightarrow}

\newcommand {\<}{\langle}
\renewcommand {\>}{\rangle}
\newcommand {\Omit}[1]{}
\newcommand{\isom}{\cong}
\newcommand{\tensor}{\otimes}

\newcommand{\C}{\mathbb{C}}
\newcommand{\Z}{\mathbb{Z}}

\newcommand{\EG}{\operatorname{EG}}
\newcommand{\Ai}{\operatorname{Ai}}
\newcommand{\bC}{\C}

\newcommand{\blob}{{\scriptscriptstyle\bullet}}

\newcommand{\todo}[1]{}
\newcommand{\cU}{\mathcal{U}}

\begin{document}
\begin{abstract}{For each integer $n\geq 2$ we describe the space of stability conditions on the derived category of the $n$-dimensional Ginzburg algebra associated to the $A_2$ quiver. The form of our results points to a close relationship between these spaces and  the Frobenius-Saito structure on the unfolding space of the $A_2$ singularity.}\end{abstract}
\maketitle

\section{Introduction}\label{sec:intro}

In this paper we study spaces of stability conditions \cite{B} on the sequence of CY$_n$ triangulated categories $\D_n$ associated to the $A_2$ quiver.
Our main result is Theorem \ref{first} below. There are several striking  features. Firstly, we obtain uniform results for all  $n\geq 2$: the space of stability conditions quotiented by the action of the spherical twists is independent of $n$, although the identification maps are highly non-trivial.  Secondly, there is a close link between our spaces of stability conditions and the Frobenius-Saito structure on  the unfolding space of the  $A_2$ singularity: in fact this structure is precisely what encodes the  identifications between our stability spaces for various  $n$. A third interesting feature is that the space of stability conditions on the derived category of the path algebra of the $A_2$ quiver arises as a kind of limit of the spaces for the categories $\D_n$ as $n\to \infty$.

\subsection{Statement of results}
For each integer $n\geq 2$ we let $\D_n=\D_{\operatorname{CY}_n}(A_2)$ denote the bounded derived category of the CY$_n$ complex Ginzburg algebra associated to the $A_2$ quiver. It is a triangulated category of finite type over $\C$, and is characterised by the following two properties:
\begin{itemize}
\item[(a)] It is CY$_n$, i.e. for any pair of objects $A,B\in \D_n$ there are natural isomorphisms
\begin{equation}
\label{cyn}\Hom^\blob_{\D_n}(A,B)\isom \Hom^\blob_{\D_n}(B,A[n])^\vee.\end{equation}
\item[(b)] It is generated by two spherical objects $S_1,S_2$ satisfying
\begin{equation}
\label{ext1}\Hom_{\D_n}^*(S_1,S_2)=\C[-1].\end{equation}
\end{itemize}
We denote by $\D_\infty$ the bounded derived category of the complex path algebra of the $A_2$ quiver. It is again a $\C$-linear triangulated category, and is characterised by the property that it is  generated by two exceptional objects $S_1,S_2$, which satisfy \eqref{ext1} and
\[\Hom_{\D_\infty}^*(S_2,S_1)=0.\]
The notation $\D_\infty$  is convenient: the  point being  that as $n$ increases, the Serre dual to the extension $S_1\to S_2[1]$ occurs in higher and higher degrees, until when $n=\infty$ it doesn't occur at all.

For each  $2\leq  n\leq \infty$ we denote by $\Stab(\D_n)$  the space of stability conditions on the category $\D_n$. We define $\Stab_*(\D_n)\subset \Stab(\D_n)$ to be the connected component containing stability conditions in which the objects $S_1$ and $S_2$ are stable of equal phase. Let $\Aut(\D_n)$ denote the group of exact $\C$-linear autoequivalences of the category $\D_n$, considered up to isomorphism of functors. We define $\Auts(\D_n)$ to be the subquotient consisting of autoequivalences which preserve the connected component $\Stab_*(\D_n)$, modulo those which act trivially on it. When $n<\infty$ the objects $S_i$ are spherical and hence define Seidel-Thomas twist functors $\Tw_{S_i}\in \Aut(\D_n)$. These autoequivalences   preserve the connected component $\Stab_*(\D_n)$, and we denote by $\Sph(\D_n)\subset \Auts(\D_n)$  the subgroup they generate.

The simple complex Lie algebra associated to the $A_2$ quiver is $\mathfrak{g}=\mathfrak{sl}_3(\C)$. Its Cartan subalgebra  can be described explicitly as
\[\h=\{(u_1,u_2,u_3)\in \C^3: \sum_i u_i=0\}.\]
The complement of the root hyperplanes is
\[\h^{\reg}=\{(u_1,u_2,u_3)\in \h: i\neq j \implies u_i\neq u_j\}.\]
There is an obvious action of the Weyl group $W=S_3$ permuting the $u_i$ which is free on $\h^{\reg}$.
The quotient $\h/W$ is isomorphic to $\C^2$, and has natural co-ordinates $(a,b)$ obtained by writing
\[p(x)=(x-u_1)(x-u_2)(x-u_3)=x^3+ax+b.\]
The image of the root hyperplanes $u_i=u_j$ is the discriminant
\[\Delta=\{(a,b)\in \C^2: 4a^3+27b^2=0\}.\]

We can now state the main result of this paper.

\begin{thm}
\label{first}
\begin{itemize}
\item[(a)] For $2\leq n<\infty$ there is an isomorphism of complex manifolds \[\Stab_*(\D_n)/\Sph(\D_n) \isom \h^{\reg}/W.\]
Under this isomorphism the central charge map $\Stab_*(\D_n)\to \C^2$ induces the multi-valued map $\h^{\reg}/W\to \C^2$ given by
\[\int_{\gamma_i} p(x)^{(n-2)/2} \, dx\]
for an appropriate basis of paths $\gamma_i$ connecting the zeroes of the polynomial $p(x)$.
\smallskip

\item[(b)] For $n=\infty$ there is an isomorphism of complex manifolds
\[\Stab(\D_\infty)\isom \h/W.\]
Under this isomorphism the central charge map $\Stab(\D_\infty)\to \C^2$ corresponds to the map $\h/W\to \C^2$ given by
\[\int_{\delta_i} e^{ p(x)}\, dx\]
for an appropriate basis of paths $\delta_i$ which approach $\infty$ in both directions along rays for which $\Re(x^3)\to -\infty$.
\end{itemize}
\end{thm}

Theorem \ref{first} gives a precise link with the Frobenius-Saito structure on the unfolding space of the $A_2$ singularity $x^3=0$. The corresponding Frobenius manifold is  precisely $M=\h/W$. The maps appearing in part (a) of our result are then the \emph{twisted period maps} of $M$ with parameter $\nu=(n-2)/2$ (see  Equation (5.11) of \cite{Du2}). The map in part (b) is given by the \emph{deformed flat co-ordinates} of $M$ with parameter $\hbar=1$ (see \cite[Theorem 2.3]{Du1}).

\subsection{Related work}
Just as we were  finishing this paper, A. Ikeda posted  the paper \cite{Ikeda}  which also proves Theorem \ref{first} (a), and indeed generalizes it to the case of the $A_k$ quiver for all $k\geq 1$. The methods we use here  are quite different however, and also yield (b),  so we feel that this paper is  also worth publishing.

As explained above, two of the most interesting features of Theorem \ref{first} are the fact that the space $\Stab_*(\D_n)/\Sph(\D_n)$ is independent of $n<\infty$, and that this space embeds in $\Stab(\D_\infty)$. At the  level of  exchange graphs such results were observed for arbitrary acyclic quivers by  one of us  with A.D. King \cite{KY}.

The $n<\infty$ case of Theorem \ref{first} was first considered by R.P. Thomas in \cite{Th}: he obtained the $n=2$ case and discussed the relationship with Fukaya categories and homological mirror symmetry.  The  $n=2$ case was also proved in \cite{Br} and generalised to arbitrary ADE Dynkin diagrams.
The $n=3$ case of Theorem \ref{first} was proved in \cite{Su}, and was extended to all Dynkin quivers of $A$ and $D$ type in \cite{BrSm}. The first statement of part (a),  that $\Stab_*(\D)\isom \h^{\reg}/W$,   was proved for all $n<\infty$ in \cite{Qiu}.

The case $n=\infty$ of Theorem \ref{first} was first considered by King \cite{King} who proved that  $\Stab(\D_\infty)\isom \C^2$. This result was obtained by several other researchers since then, and a proof was written down in \cite{Qiu}. The more precise statement of Theorem \ref{first} (b) was conjectured by A. Takahashi \cite{Ta}.

Since the first version of this paper was posted, several  generalizations and extensions of Theorem \ref{first} have appeared. The following seem particularly noteworthy. In \cite{HKK} a  general result  relating stability conditons on Fukaya categories of surfaces to spaces of  quadratic differentials with exponential singularities is proved. This includes Theorem \ref{first} (b) as a very special case.

In \cite{IY} a new notion of q-stability conditions is introduced, in terms of which one can make sense of the statement of Theorem \ref{first} when the Calabi-Yau dimension $n\geq 2$ is replaced with an arbitrary $s\in \bC$ with $\Re(s)\geq 2$. This allows one to see all twisted periods of the Frobenius manifold via central charges.

In \cite{Mc}, an analogue of Theorem \ref{first} for the Kronecker quiver is proved. In this case the relevant Frobenius manifold is the quantum cohomology of $\P^1$. The case of quivers of type affine $A_n$ was considered in \cite{W}. A general framework for such results, involving Fukaya categories of surfaces, quadratic differentials and Hurwitz spaces is explained in \cite{IY2}.

\subsection*{Acknowledgements.} We thank Alastair King and Caitlin McAuley for many useful conversations on the topic of this paper.
Qiu is is supported by Beijing Natural Science Foundation (Z180003).

\section{Autoequivalences and t-structures}

In this section we describe the principal components of the exchange graphs of the categories $\D_n=\D_{\operatorname{CY}_n}(A_2)$ and study the action of the group of reachable autoequivalences. We start by recalling some general definitions concerning tilting (see  \cite[Section 7]{BrSm} for more details).

\subsection{}
Let $\D$ be a triangulated category. We shall be concerned with bounded t-structures on $\D$.
Any such t-structure  is determined by its heart $\A\subset \D$, which is a full abelian subcategory. We  use the term  \emph{heart} to mean the heart of a bounded t-structure.
A heart  will be called \emph{finite-length} if it is   artinian and noetherian as an abelian category.

We say that a pair of  hearts $(\A_1,\A_2)$ in $\D$ is a \emph{tilting pair}  if the equivalent conditions
\[\A_2\subset \langle \A_1,\A_1[-1] \rangle, \quad \A_1\subset \langle\A_2[1],\A_2\rangle\]
are satisfied. Here the angular brackets  signify the extension-closure operation
We also say that $\A_1$ is a \emph{left tilt} of $\A_2$, and that $\A_2$ is a \emph{right tilt} of $\A_1$. Note that $(\A_1,\A_2)$ is  a tilting pair precisely if  so is $(\A_2[1],\A_1)$.

If $(\A_1,\A_2)$ is a tilting pair in $\D$, then  the subcategories \[\T=\A_1\cap \A_2[1], \quad  \FF=\A_1\cap \A_2\] form a torsion pair $(\T,\FF)\subset \A_1$. Conversely, if $(\T,\FF)\subset \A_1$ is a torsion pair, then the subcategory $\A_2=\langle \FF, \T[-1]\rangle$ is a heart, and the pair $(\A_1,\A_2)$ is a tilting pair.

%

A special case of the tilting construction will be particularly important.
Suppose that $\A$ is a finite-length heart and $S\in\A$ is a simple object.
 Let
$\langle S \rangle\subset \A$ be the full subcategory consisting of objects $E\in\A$ all
of whose simple factors are isomorphic to $S$. Define full subcategories
\[ \smash{^\perp}{S}=\{E\in\A:\Hom_{\A}(E,S)=0\}, \qquad S^\perp=\{E\in \A:\Hom_{\A}(S,E)=0\}.\]
One can either view $\langle S \rangle$ as the
torsion-free part of a torsion pair on $\A$, in which case the torsion part is $\smash{^\perp}{S}$, or as  the torsion part,  in which case the torsion-free
part is $S^\perp$.
We can then define tilted
 hearts
\[ \mu^-_S (\A)=\langle S[1],\smash{^\perp}{S}\rangle,\qquad \mu^+_S (\A) = \langle S^\perp, S[-1]\rangle,\]
which we refer to as the left and right tilts of the heart $\A$ at the simple $S$.
They fit into tilting pairs $(\mu^-_S (\A),\A)$ and $(\A,\mu^+_S (\A))$. Note the relation
\begin{equation*}
\mu^+_{S[1]}\circ \mu^-_S(\A)=\A.\end{equation*}
The \emph{exchange graph} $\EG(\D)$ is the graph with  vertices the finite-length hearts in $\D$ and edges corresponding to simple tilts. The group $\Aut(\D)$ of triangulated autoequivalences of $\D$ acts on this graph in the obvious way: an auto-equivalence $\Phi\in \Aut(\D)$ sends a finite-length heart $\A\subset \D$ to the finite-length heart $\Phi(\A)\subset \D$.

\begin{figure}
\begin{center}
\begin{tikzpicture}[scale=0.45]
\draw (0,4)--(13,4);
\draw (0,9)--(13,9);
\draw (1.5,4)--(1.5,9);
\draw (4,4)--(4,9);
\draw (9.25,4)--(9.25,9);
\draw (11.75,4)--(11.75,9);
\draw (2.75,6.5) node {$S[1]$};
\draw (6.5,6.5) node {$^\perp\!S$};
\draw (10.5,6.5) node {$S$};
\draw [decorate, decoration={brace,amplitude=5pt}] (4,9.5)--(11.75,9.5)
node [midway, above=6pt] {$\A$};
\draw [decorate, decoration={brace,amplitude=5pt}] (9.25,3.5)--(1.5, 3.5)
node [midway, below=6pt] {$\mu_S^-(\A)$};
\end{tikzpicture}
\quad\quad
\begin{tikzpicture}[scale=0.45]
\draw (0,4)--(13,4);
\draw (0,9)--(13,9);
\draw (1.5,4)--(1.5,9);
\draw (4,4)--(4,9);
\draw (9.25,4)--(9.25,9);
\draw (11.75,4)--(11.75,9);
\draw (2.75,6.5) node {$S$};
\draw (6.5,6.5) node {$S^\perp$};
\draw (10.5,6.5) node {$S[-1]$};
\draw [decorate, decoration={brace,amplitude=5pt}] (1.5,9.5)--(9.25,9.5)
node [midway, above=6pt] {$\A$};
\draw [decorate, decoration={brace,amplitude=5pt}] (11.75,3.5)--(4, 3.5)
node [midway, below=6pt] {$\mu_{S}^+(\A)$};
\end{tikzpicture}
\caption{Left and right tilts of a heart.}\label{fig:mut-hearts}
\end{center}
\end{figure}

\subsection{}   For each $2\leq n\leq \infty$ we define the triangulated category $\D_n$ as in the Introduction. It is the bounded derived category of the CY$_n$ Ginzburg  algebra associated to the A$_2$ quiver \cite{Ginz}. This category contains two distinguished objects $S_1$ and $S_2$ corresponding to the vertices of the quiver,  and a canonical   heart \[\A_n=\<S_1,S_2\>\subset \D_n,\] which is the extension-closed subcategory generated by these objects. The heart $\A_n$ has finite-length, and hence defines a point of the exchange graph $\EG(\D_n)$; we denote by $\EG^\circ(\D_n)$  the connected  component containing this point. We call  $\EG^\circ(\D_n)$ the \emph{principal component} of the exchange graph, and refer to the hearts defined by  its  vertices as \emph{reachable} hearts. We say that a heart $\A\subset \D_n$ is \emph{full} if it is equivalent to $\A_n$ as an abelian category.


\begin{remark}
When $n>2$, the canonical heart $\A_n$ is  equivalent to the category $\Rep(A_2)$ of representations of the $A_2$ quiver;
besides the simple objects $S_1$ and $S_2$, it contains one more indecomposable object which we denote by $E$; there is a short exact sequence
\begin{equation}
\label{ses}
0\lra S_2\lra E\lra S_1\lra 0.\end{equation}
When $n=2$, the canonical heart  $\A_2$ is equivalent to the category of representations of the preprojective algebra of the $A_2$ quiver; besides  $E$ there is another non-simple indecomposable object fitting into a short exact sequence
\begin{equation}
\label{sestwo}
0\lra S_1\lra F\lra S_2\lra 0.\end{equation}
For this reason, the case $n=2$ is slightly special, and since our main result Theorem \ref{first} is already known (and is considerably easier to prove) in this case \cite{Th,Br}, in what follows we shall restrict to the case $n\geq 3$.
\end{remark}

A triangulated  autoequivalence of $\D_n$  is called \emph{reachable} if its action on $\EG(\D_n)$ preserves the connected component $\EG^\circ(\D_n)$.   An autoequivalence is called \emph{negligible} if it is reachable and acts by the identity on $\EG^\circ(\D_n)$. It is easy to see that an autoequivalence $\Phi\in \Aut(\D_n)$ is negligible precisely if $\Phi(S_1)\isom S_1$ and $\Phi(S_2)\isom S_2$.  We write $\Auts(\D_n)$ for the subquotient of the group $\Aut(\D)$ consisting of reachable autoequivalences,  modulo negligible autoequivalences. We will show that this agrees with the definition given in the Introduction later: see Section \ref{bking} and Prop. \ref{fundom}.

In the case $n<\infty$ an important role will be played by spherical  twist functors \cite{ST}. Recall  that an object $S\in \D_n$ is called \emph{spherical} if
\[\Hom^{\blob}_{\D_n}(S,S)=\bC\oplus \bC[-n].\]
Any such object defines an autoequivalence  $\Tw_S\in\Aut(\D_n)$ called a \emph{spherical twist}. This has the property that for each object $E\in \D_n$ there is a triangle
\[\Hom_{\D_n}^{\blob}(S,E)\tensor S\lra E\lra \Tw_S(E),\]
where the first arrow is the evaluation map.
In particular, the distinguished objects $S_i\in \D_n$ are spherical, and hence define twist functors $\Tw_{S_i}\in \Aut(\D_n)$.

\begin{lemma}
\label{little}
Take $2\leq n<\infty$, and  define the following autoequivalences of $\D_n$:
\[\Sigma=(\Tw_{S_1} \circ\Tw_{S_2})[n-1], \qquad \Upsilon=(\Tw_{S_2}\circ\Tw_{S_1}\circ\Tw_{S_2})[2n-3].\]
Then we have
\[\Sigma(S_1,E,S_2)=(S_2[1],S_1,E), \qquad \Upsilon(S_1,S_2)=(S_2,S_1[n-2]).\]
\end{lemma}

\begin{proof}
The  defining properties \eqref{cyn} and  \eqref{ext1} of the category $\D_n$, together with the short exact sequence \eqref{ses},  implies that
$\Tw_{S_1}(S_2)=E$. In particular  $E$ is also spherical. Applying the long exact sequence in cohomology to the short exact sequence \eqref{ses}, and using  the fact that the objects $S_i$ are spherical, shows that \begin{equation}
\label{fk}\Hom^\blob_{\D_n}(S_1,E)=\bC[-n], \qquad \Hom^\blob_{\D_n}(E,S_2)=\C[-n].\end{equation}
Using the CY$_n$ property we therefore obtain  the identities
\[\Tw_{S_1}(S_2)=E, \qquad \Tw_E(S_1)=S_2[1], \qquad \Tw_{S_2}(E)=S_1.\]

For any spherical object $S$ there is an identity $\Tw_{S}(S)=S[1-n]$, and for any pair of spherical objects there is a relation
\[\Tw_{S_1}\circ \Tw_{S_2}=\Tw_{\Tw_{S_1}(S_2)}\circ  \Tw_{S_1}.\]
Thus we can write  $\Sigma=(\Tw_E\circ \Tw_{S_1})[n-1]$. This implies that
\[\Sigma(S_1)=\Tw_E(S_1)=S_2[1], \qquad \Sigma(S_2)=\Tw_{S_1}(S_2)=E,\]
 and it follows that $\Sigma(E)$ is the unique nontrivial extension of these two objects, namely $S_1$. This proves the first claim.

Moving on to the second identity, we use the braid relation  \[\Tw_{S_1}\circ\Tw_{S_2}\circ\Tw_{S_1}=\Tw_{S_2}\circ\Tw_{S_1}\circ\Tw_{S_2}\]
proved by Seidel and Thomas  \cite{ST}. This implies that
\[\Upsilon(S_1)=\Sigma (S_1[-1])=S_2, \quad \Upsilon(S_2)=\Tw_{S_2}(E[n-2])=S_1[n-2],\]
which completes the proof.
\end{proof}

\subsection{}The following description of the tilting operation in $\D_n$ is the combinatorial underpinning of our main result.

\begin{prop}
\label{tilts}
Take $3\leq n< \infty$, and consider the hearts obtained by performing simple tilts of the standard heart $\A_n\subset \D_n$. We have
\begin{itemize}
\item[(a)]The left tilt of $\A_n$ at the simple $S_2$ is another full heart:
\[\A_n=\<S_1,S_2\> \to  \<S_2[1],E\>=\Sigma(\A_n).\]

\item[(b)]Repeated left tilts at appropriate shifts of $S_1$ gives a sequence of hearts
\[\A_n=\<S_1,S_2\>\to\<S_1[1],S_2\>\to\<S_1[2],S_2\>\to \cdots\to\<S_1[n-2],S_2\>=\Upsilon(\A_n).\]
\end{itemize}
\end{prop}

\begin{proof}
This can be found in \cite[Proposition~5.4]{KY}, but for the convenience of the reader we give a proof here. Note that since $E\in \smash{^\perp}{S_2}$ we have $\<S_2[1],E\>=\Sigma(\A_n)\subset \mu^-_{S_2}(\A_n).$
But it is a standard fact, and easily proved, that if one heart is contained in another then they are equal. This gives (a). For the first step in part (b) note that since the only indecomposables in $\A_n$ are $S_1,S_2$ and $E$, we have $\smash{^\perp}{S_1}=\langle S_2\rangle$. It follows that $\mu^-_{S_1}(\A_n)=\langle S_1[1],S_2\rangle$. We have \[\Hom^1_{\D_n}(S_1[k],S_2)=0=\Hom^1_{\D_n}(S_2,S_1[k]),\qquad 0<k<n-2,\]
so if $n>3$ the heart $\langle S_1[1],S_2\rangle$  is semi-simple, with simple objects $S_1[1]$ and $S_2$. It follows that its left tilt with respect to the simple $S_1[1]$ is the subcategory $\<S_1[2],S_2\>$. Repeating the argument we obtain the given sequence of hearts.
\end{proof}

As explained in the proof of Prop. \ref{tilts}, when $n>3$ each of the intermediate hearts in the sequence in (b) is semi-simple, and in particular non-full. The semi-simplicity implies that
\[\mu^\pm_{S_2}\<S_1[k],S_2\>=\<S_1[k],S_2[\pm 1]\>,\]
so that tilting these intermediate hearts with respect to $S_2$ gives hearts of the same kind up to shift.

\begin{lemma}
\label{bl}
Take $3\leq n< \infty$. Then  every reachable heart is of the form
\[ \Phi(\<S_1[k],S_2\>)\subset \D_n,\]
where $\Phi$  is a reachable autoequivalence, and  $0\leq k\leq (n-2)/2$.
\end{lemma}

\begin{proof}
What we proved above shows that every time we tilt a heart of the form $\<S_1[k],S_2\>$ with $0\leq k<n-2$ we obtain another heart of the same form, up to the action of some autoequivalence, which is necessarily reachable. Since tilting commutes with autoequivalences, in the sense that
\[\mu^\pm _{\Phi(S)}(\Phi(\A))=\Phi(\mu^\pm _S(\A)),\]
it follows that any heart is of the  given form, with $0\leq k<n-2$. To complete the proof, note that if $k_1+k_2=n-2$, the autoequivalence $\Upsilon$ exchanges the hearts $\<S_1[k_i],S_2\>$ up to shift.
\end{proof}


\begin{lemma}
\label{powers}
Take $3\leq n<\infty$.
\begin{itemize}
\item[(a)] The autoequivalences $\Sigma, \Upsilon$ and $[1]$ are all reachable.
\item[(b)] The spherical twists $\Tw_{S_i}$ are reachable.
\item[(c)] In the group $\Auts(\D_n)$ there are relations
\begin{equation*}
\Sigma^3=[1], \quad \Upsilon^2=[n-2].\end{equation*}
\end{itemize}
\end{lemma}

\begin{proof}
The reachability of $\Sigma$ and $\Upsilon$ is immediate from Prop. \ref{tilts}. Lemma \ref{little} shows that $\Sigma^3(\A_n)=\A_n[1]\subset \D_n$. Thus the shift functor $[1]$ is also reachable, which gives (a). Part (b) then follows from  the  relations
\[(\Sigma\circ\Tw_{S_1})[n-2]=\Upsilon=(\Tw_{S_2}\circ\Sigma)[n-2]\]
For part (c) consider the autoequivalence $\Sigma^3[-1]$. Since it fixes the objects $S_1,S_2$, it  is negligible, and hence defines the identity element in $\Auts(\D_n)$. The same argument applies to $\Upsilon^2[2-n]$.
\end{proof}

\begin{lemma}
For $3\leq n<\infty$ the action of $\Auts(\D_n)$ on the set of full reachable hearts is free and transitive.
\end{lemma}

\begin{proof}
The transitivity follows from Lemma \ref{bl}, since for $k$ in the given range the heart $\<S_1[k],S_2\>$ is full only for $k=0$. For the freeness it is therefore enough to consider autoequivalences which preserve the standard heart $\A_n\subset \D_n$. Any such autoequivalence $\Phi$ must preserve the simple objects $S_1,S_2$ of $\A_n$, and it cannot exchange them since it preserves the $\Ext$-groups between them, which are asymmetric for $n\geq 3$. Thus $\Phi(S_i)=S_i$ and it follows that $\Phi$ defines the identity element of $\Aut_*(\D_n)$. \end{proof}

Note that when $n=3$ it follows from Lemma \ref{bl}  that all reachable hearts are full. Thus in this case the principal component $\EG^\circ(\D_3)$ is a torsor for the group $\Auts(\D_3)$.

 \subsection{}
We denote by $\Br_3$ the Artin braid group of the $A_2$ root system;  in the notation of the Introduction it is the fundamental group of the quotient $\h^{\reg}/W$. More concretely, $\Br_3$ is the standard braid group on three strings, and has a presentation
\[\Br_3=\< \sigma_1,\sigma_2: \sigma_1\sigma_2 \sigma_1=\sigma_2 \sigma_1 \sigma_2\>.\]
The centre of $\Br_3$ is generated by the element $\tau=(\sigma_1 \sigma_2)^3$.

\begin{figure}\centering
\begin{tikzpicture}[scale=0.77]
\path (0,0) coordinate (O);
\path (0,6) coordinate (S1);
\draw
    (O)[Periwinkle,dotted,thick]   circle (6cm);
\draw
    (S1)[\graphc,thick]
    \foreach \j in {1,...,3}{arc(360/3-\j*360/3+180:360-\j*360/3:10.3923cm)} -- cycle;
\draw
    (S1)[\graphc,semithick]
    \foreach \j in {1,...,6}{arc(360/6-\j*360/6+180:360-\j*360/6:3.4641cm)} -- cycle;
\draw
    (S1)[\graphc]
    \foreach \j in {1,...,12}{arc(360/12-\j*360/12+180:360-\j*360/12:1.6077cm)}
        -- cycle;
\foreach \j in {1,...,3}
{\path (-90+120*\j:.7cm) node[\vertexc] (v\j) {$\bullet$};
 \path (-210+120*\j:.7cm) node[\vertexc] (w\j) {$\bullet$};
 \path (-90+120*\j:2.2cm) node[\vertexc] (a\j) {$\bullet$};
 \path (-90+15+120*\j:3cm) node[\vertexc] (b\j) {$\bullet$};
 \path (-90-15+120*\j:3cm) node[\vertexc] (c\j) {$\bullet$};}
\foreach \j in {1,...,3}
{\path[->,>=stealth] (v\j) edge[\thirdc,bend left,thick] (w\j);
 \path[->,>=stealth] (a\j) edge[\thirdc,bend left,thick] (b\j);
 \path[->,>=stealth] (b\j) edge[\thirdc,bend left,thick] (c\j);
 \path[->,>=stealth] (c\j) edge[\thirdc,bend left,thick] (a\j);}
\foreach \j in {1,...,3}
{\path (60*\j*2-1-60:3.9cm) node[\vertexc] (x1\j) {$\bullet$};
 \path (60*\j*2+1-120:3.9cm) node[\vertexc] (x2\j) {$\bullet$};
}
\foreach \j in {1,...,3}
{\path[->,>=stealth] (v\j) edge[\halfc,bend left,thick] (a\j);
 \path[->,>=stealth] (a\j) edge[\halfc,bend left,thick] (v\j);
 \path[->,>=stealth] (b\j) edge[\halfc,bend left,thick] (x1\j);
 \path[->,>=stealth] (c\j) edge[\halfc,bend left,thick] (x2\j);
 \path[->,>=stealth] (x1\j) edge[\halfc,bend left,thick] (b\j);
 \path[->,>=stealth] (x2\j) edge[\halfc,bend left,thick] (c\j);}
\draw (-30:6cm) node[below right] {$0$}
    (90:6cm) node[above] {$1$}
    (-90:6cm) node[below] {$-1$}
    (210:6cm) node[below left] {$\infty$};
\end{tikzpicture}
\caption{The  projective exchange graph of $\D_3$ drawn on the hyperbolic disc. The action of $\P\Auts(\D_3)$ corresponds to the standard action of $\PSL(2,\Z)$  on the disc.}
\label{fig:G3}
\end{figure}

\begin{figure}\centering
\begin{tikzpicture}[scale=0.52]
\path (0,0) coordinate (O);
\path (0,6) coordinate (S1);
\draw
    (O)[Periwinkle,dotted,thick]   circle (6cm);
\draw
    (S1)[\graphc,thick]
    \foreach \j in {1,...,3}{arc(360/3-\j*360/3+180:360-\j*360/3:10.3923cm)} -- cycle;
\draw
    (S1)[\graphc,semithick]
    \foreach \j in {1,...,6}{arc(360/6-\j*360/6+180:360-\j*360/6:3.4641cm)} -- cycle;
\foreach \j in {1,...,3}
{\path (-90+120*\j:1.6cm) node[\vertexc]  {$\bullet$};
 \path (-210+120*\j:1.6cm) node[\vertexc] {$\bullet$};
 \path (-90+30+120*\j:3.4641cm) node[\vertexc] {$\bullet$};
 \path (-90-30+120*\j:3.4641cm) node[\vertexc] {$\bullet$};}
\foreach \j in {1,...,3}
{\path (-90+120*\j:1.6cm) node[\vertexc] (v\j) {};
 \path (-210+120*\j:1.6cm) node[\vertexc] (w\j) {};
 \path (-90+30+120*\j:3.4641cm) node[\vertexc] (b\j) {};
 \path (-90-30+120*\j:3.4641cm) node[\vertexc] (c\j) {};}
\foreach \j in {1,...,3}
{\path (v\j) edge[\thirdc,bend left=15,thick] (w\j);
 \path (v\j) edge[\thirdc,bend left=15,thick] (b\j);
 \path (b\j) edge[\thirdc,bend left=25,thick] (c\j);
 \path (c\j) edge[\thirdc,bend left=15,thick] (v\j);}
\end{tikzpicture}
\;
\begin{tikzpicture}[scale=0.52]
\path (0,0) coordinate (O);
\path (0,6) coordinate (S1);
\draw
    (O)[Periwinkle,dotted,thick]   circle (6cm);
\draw
    (S1)[\graphc,thick]
    \foreach \j in {1,...,3}{arc(360/3-\j*360/3+180:360-\j*360/3:10.3923cm)} -- cycle;
\draw
    (S1)[\graphc,semithick]
    \foreach \j in {1,...,6}{arc(360/6-\j*360/6+180:360-\j*360/6:3.4641cm)} -- cycle;
 \foreach \j in {1,...,3}
{\path (-90+120*\j:.6cm) node[\vertexc]   {$\bullet$};
 \path (-210+120*\j:.6cm) node[\vertexc] {$\bullet$};
 \path (-90+120*\j:1.6cm) node[\vertexc] {$\bullet$};
 \path (-90+120*\j:2.4cm) node[\vertexc] {$\bullet$};
 \path (-90+17+120*\j:3cm) node[\vertexc]{$\bullet$};
 \path (-90-17+120*\j:3cm) node[\vertexc]{$\bullet$};}
  \foreach \j in {1,...,3}
{\path (-90+120*\j:.6cm) node[\vertexc]  (v\j) {};
 \path (-210+120*\j:.6cm) node[\vertexc] (w\j){};
 \path (-90+120*\j:1.6cm) node[\vertexc] (x\j){};
 \path (-90+120*\j:2.4cm) node[\vertexc] (a\j){};
 \path (-90+17+120*\j:3cm) node[\vertexc] (b\j){};
 \path (-90-17+120*\j:3cm) node[\vertexc] (c\j){};}
\foreach \j in {1,...,3}
{\path (v\j) edge[\thirdc,bend left,thick] (w\j);
 \path (a\j) edge[\thirdc,bend left=10,thick] (b\j);
 \path (b\j) edge[\thirdc,bend left=20,thick] (c\j);
 \path (c\j) edge[\thirdc,bend left=10,thick] (a\j);}
\foreach \j in {1,...,3}
{\path (v\j) edge[\halfc,bend left,thick] (x\j);
 \path (x\j) edge[\halfc,bend left,thick] (v\j);}
\foreach \j in {1,...,3}
{\path (x\j) edge[\halfc,bend left,thick] (a\j);
 \path (a\j) edge[\halfc,bend left,thick] (x\j);}
\foreach \j in {1,...,3}
{ \path (-90+30+120*\j:3.4641cm) node[\vertexc] {$\bullet$};
 \path (-90-30+120*\j:3.4641cm) node[\vertexc] {$\bullet$};}
\foreach \j in {1,...,3}
{\path (-90+30+120*\j:3.4641cm) node[\vertexc] (d\j) {};
 \path (-90-30+120*\j:3.4641cm) node[\vertexc] (e\j) {};}
 \foreach \j in {1,...,3}
{\path (b\j) edge[\halfc,bend left,thick] (d\j);
 \path (c\j) edge[\halfc,bend left,thick] (e\j);
 \path (d\j) edge[\halfc,bend left,thick] (b\j);
 \path (e\j) edge[\halfc,bend left,thick] (c\j);}
\end{tikzpicture}
\caption{Similar pictures of the projective exchange graphs of $\D_2$ and $\D_4$ (orientations omitted). As before, the action of $\P\Auts(\D_n)\isom \PSL(2,\Z)$ corresponds to the standard one.
 }
\label{fig:G24}
\end{figure}

\begin{prop}
\label{bug}
Take $3\leq n<\infty$.
\begin{itemize}
\item[(a)] The group $\Auts(\D_n)$ is generated by $\Sigma, \Upsilon$ and the shift functor $[1]$.

\item[(b)] The group $\Auts(\D_n)$ is generated by the subgroup $\Sph(\D_n)$ together with  $[1]$.

\item[(c)] There is an isomorphism $\Br_3\isom \Sph(\D_n)$ sending the generator $\sigma_i$ to $\Tw_{S_i}$.

\item[(d)] The isomorphism in (c) sends the central element $\tau$  to $[4-3n]$.

\item[(e)]  The smallest power of $[1]$ contained in $\Sph(\D_n)$ is $[3n-4]$. Thus there is a short exact sequence
\[1\lra \Sph(\D_n)\lra \Auts(\D_n)\lra \mu_{3n-4}\lra 1.\]

\end{itemize}
\end{prop}

\begin{proof}
Part (a) follows from the explicit description of tilts given in Prop. \ref{tilts}, since any element of $\Auts(\D_n)$ takes the canonical heart $\A$ to a full reachable heart. Part (b) is then immediate from the definitions of $\Sigma$ and $\Upsilon$.
It was  proved in \cite{ST} that there is an injective group homomorphism  $\rho\colon \Br_3\to \Aut(\D_n)$ sending the generator $\sigma_i$ to the twist functor $\Tw_{S_i}$. By Lemma \ref{powers}(b) this induces a surjective homomorphism $\bar{\rho}\colon \Br_3\to \Sph(\D_n)$. It is immediate from Lemma \ref{powers} and the definition of  $\Upsilon$ that  $\bar{\rho}(\tau)=[4-3n]$, which gives (d). Note that if $g\in \Br_3$ satisfies $\bar{\rho}=1$, then the autoequivalence $\rho(g)$ is negligible, so    commutes with the autoequivalences $\Tw_{S_i}$, which shows that $g$ must be central. But then $g$ must be a power of $\tau$ which,  by (d), implies that $g$ is the identity.  This completes the proof of (c).
Part (e) again follows from the fact that $\tau$ generates the centre of $\Br_3$, since any shift $[d]$ lying in the subgroup $\Sph(\D)\subset \Aut_*(\D_n)$ is necessarily central.
\end{proof}

\subsection{}It will be useful to introduce the quotient group
\[\P\Auts(\D_n)=\Auts(\D_n)/[1].\]
When $3\leq n<\infty$ the natural action of auto-equivalences on the Grothendieck group $K_0(\D)\isom \Z^{\oplus 2}$ induces a group homomorphism
\begin{equation}
\label{sheff}\rho\colon \P\Auts(\D_n)\to \operatorname{PGL}(K_0(\D)).\end{equation}
Taking the basis $([S_1],[S_2])\subset K_0(\D)$ we can identify the target of this map with $\operatorname{PGL}(2,\Z)$. From the definition of the twist functors we have
\[\rho(\Tw_{S_1})=\mat{(-1)^{n+1}}{1}{0}{1}, \qquad \rho(\Tw_{S_2})=\mat{1}{0}{(-1)^n}{(-1)^{n+1}}.\]
It is a standard fact\todo{Reference} that  for $n$ odd, the map $\Br_3\to \operatorname{PGL}(2,\Z)$ sending the generators $\sigma_1,\sigma_2$ to these matrices induces an isomorphism $\Br_3/\<\tau\>\isom \operatorname{PSL}(2,\Z) < \operatorname{PGL}(2,\Z)$.  It then follows from Prop. \ref{bug}  that the map \eqref{sheff} is an isomorphism onto its image.  Note that
\begin{equation}
\label{5am} \rho(\Sigma)=  \mat{0}{-1}{1}{1}, \qquad \rho(\Upsilon)= \mat{0}{1}{(-1)^n}{0}.\end{equation}
These elements have order 3 and 2 respectively.   When considering spaces of stability conditions in Section  \ref{curry} we will see the actions of these autoequivalences on the dual space $\Hom_{\Z}(\Gamma,\C)$, which are given by the transposes of the same matrices.

\subsection{}The autoequivalence group of the category $\D_\infty$ is much simpler.
\begin{prop}
\label{easy}There is an equality
$\Aut_*(\D_\infty)=\Aut(\D_\infty)$. Moreover

\begin{itemize}
\item[(a)]
The group $\Aut(\D_\infty)\isom \Z$ with the Serre functor $\Sigma$ being a generator.
\item[(b)] There is a relation $\Sigma^3=[1]$.
\end{itemize}
\end{prop}

\begin{proof}
This is easy and well-known. The Auslander-Reiten quiver for $\D_\infty$ is an infinite sequence
\[\cdots \to E[-1]\to S_1[-1]\to  S_2\to E\to S_1\to S_2[1]\to E[1]\to \cdots\]
and $\Sigma$ moves along this to the right by one place.
\end{proof}

It follows from this result that  $\P\Aut(\D_\infty)\isom \mu_3$. Note that our use of the symbol $\Sigma$ in Prop. \ref{easy} is reasonably consistent with our earlier use for the category $\D_n$ for $n<\infty$. For example, Prop. \ref{tilts} (a)  continues to hold when $n=\infty$. Note also that Prop. \ref{tilts} (b) becomes an infinite chain of tilts of non-full hearts in this case.


\section{Conformal maps}
\label{conf}

In this section we describe some explicit conformal maps which will be the analytic ingredients in the proof of our main result. We set $\omega=\exp(2\pi i /3)$. We introduce the M{\"o}bius transformation $T(z)=-(z+1)/z$ of order 3 defined by the transpose of the matrix $\rho(\Sigma)$ appearing in \eqref{5am}.  Consider the
unit circle
$C_0=\{z\in \bC: |z|=1\}$.
 It is easily checked that (cf. Figure~\ref{fig:circles})
\begin{equation}
\label{drn}C_{\pm}=T^{\pm 1}(C_0)=\{z\in \bC: |z^{\pm 1}+1|=1\} ,\end{equation}
and  that $C_-=\{z\in \C: \Re(z)=-\tfrac{1}{2}\}$. These circles are illustrated in Figure~\ref{fig:circles}, together with the region cut out by the inequalities
\begin{equation}\label{eq:ineq}
    |z+1|>1, \qquad |z+1|>|z|.
\end{equation}
\begin{figure}[h]\centering
\begin{tikzpicture}[scale=2]
\draw[white,fill=gray!14] (-.5,-1.5) rectangle (2,1.7);
\draw[thick,fill=white] (-1,0) circle (1) (1,.8) node{$C_0$} (-2,.8) node{$C_+$} (-.5,1.7) node[above] {$C_-$};
\draw[thick] (0,0) circle (1);
\draw[dotted,->,>=stealth] (-3,0) -- (2,0) node[right]{$x$};
\draw[dotted,->,>=stealth] (0,-1.5) -- (0,1.7) node[above]{$y$};
\draw[thick] (-.5,-1.5) -- (-.5,1.7);
\end{tikzpicture}
\caption{The circles $C_0$ and $C_\pm$.}\label{fig:circles}
\end{figure}

\subsection{}

Fix an integer $3\leq n<\infty$ and set $\nu=(n-2)/2$.
Consider the domain $R_n\subset \C\subset \P^1$ depicted in Figure \ref{fig:A2N}. It is bounded by the line $\Re(z)=-\nu$ and by the curves $\ell_\pm$ which are the images under the map $z\mapsto (1/\pi i)\log(z)$ of the arcs of the circles $C_\pm$
connecting $0^{\pm 1}$ and $\omega$, where $0^{-1}=\infty$. We also consider splitting  $R_n$ into two halves $R_n^\pm$ by dividing it along the line $\Im(z)=0$, and we take $R_n^-$ to be the part lying below
the real axis. Note that the boundary of the domain $R_n^-$ has three corners, namely $-\nu, \infty, \tfrac{2}{3}$, and these occur in anti-clockwise order.
Carath{\'e}odory's extension of the Riemann mapping theorem, e.g. \cite[Theorem 2.8.8]{BG}, ensures that there is a unique biholomorphism
\[f_n\colon\H\to  R^-_n \]
which extends homeomorphically over the boundary of the upper half plane $\H\subset \P^1$, and sends $(0,1,\infty)$ to $(-\nu, \infty, \tfrac{2}{3})$.

\begin{figure}[h]\centering
\begin{tikzpicture}[scale=.35]
\path (0,0) coordinate (O);
\path (4,0) coordinate (A);
\path (3+0.2,10) coordinate (m1);
\path (3+0.2,-10) coordinate (m2);
\draw[fill=gray!14] (m1)
    .. controls +(-90:3) and +(120:1) .. (A)
    .. controls +(-120:1) and +(90:3).. (m2)
   to [out=180,in=0] (-15,-10)
   .. controls +(90:3) and +(-90:1) ..  (-15,10) ;
\draw[white,thick] (m2) edge (-15,-10);
\draw[dotted,->,>=stealth] (-20,0) -- (8,0) node[right]{$x$};
\draw[dotted,->,>=stealth] (0,-10) -- (0,10) node[above]{$y$};
\draw (0,0) node[below left]{$0$};
\path[dotted] (6,-10) edge (6,10);
\path[dotted] (3,-10) edge (3,10);
\path[thick] (-15,0) edge (-15,10);
\path[thick] (-15,-10) edge (-15,0);
\draw[red,thick] (A) edge (6,0);
\draw[red,thick] (m1) .. controls +(-90:3) and +(120:1) .. (A);
\draw[red,thick] (A)  .. controls +(-120:1) and +(90:3).. (m2);
\draw (6,0) node[below right] {$1$};
\path (5.5,0) node (a) {$$};
\path (3.6,1.5) node (b) {$$};
\path (3.6,-1.5) node (c) {$$};
\path[->,orange, bend right] (a) edge (b);
\path[->,orange, bend right] (b) edge (c);
\path[->,orange, bend right] (c) edge (a);
\path (-15,1.5) node (d) {$$};
\path (-15,-1.5) node (e) {$$};
\path[orange, bend right=45] (d) edge (-16.5,0);
\path[->,orange, bend right=45] (-16.5,0) edge (e);
\path[orange, bend right=45] (e) edge (-15+1.5,0);
\path[->,orange, bend right=45] (-15+1.5,0) edge (d);
\path (4,6) node {$l_+$};\path (4,-6) node {$l_-$};
\draw[fill=black] (0,0) circle (.05);
\draw[fill=white] (6,0) circle (.1);
\draw[fill=red] (4,0) circle (.2);
\draw[fill=black] (-15,0) circle (.2);
%
%
\path (4,0) node[below right] {\tiny{$\frac{2}{3}$}};
\path (-15,0) node[below right] {\tiny{$\frac{2-n}{2}$}};
\end{tikzpicture}
\caption{The region $R_n$.}\label{fig:A2N}
\end{figure}

\begin{prop}
\label{endy}
The function $f_n$ can be  written explicitly as
\[f_n(t)=\frac{1}{\pi i}\log\bigg(\frac{\phi^{(2)}_n(a,b)}{\phi^{(1)}_n(a,b)}\bigg),\]
where $t=-27b^2/4a^3$, and
\begin{gather}\label{eq:twisted p}
    \phi_n^{(i)}(a,b)=\int_{\gamma_i} (x^3+ax+b)^{\frac{n-2}{2}} dx,
\end{gather}
for appropriately chosen paths $\gamma_i$ connecting zeroes of the integrand.
\end{prop}

Note that the given expression for $f_n$ depends only on the quantity $t$ because rescaling $(a,b)$ with weights $(4,6)$ rescales both the functions $\phi_n^{(i)}$ with weight $3(n-2)+2=3n-4$,
 and hence leaves their ratio unchanged.
We will split the proof of Proposition \ref{endy} into two steps. In the first, we show that the function $g_n(t)=\exp(\pi i f_n(t))$ is given by the ratio of solutions to a hypergeometric equation, and  in the second we show that  the periods \eqref{eq:twisted p} satisfy the same equation.

\subsection{}

The first part of the proof of Proposition \ref{endy}  is a minor extension of the usual proof of the Schwarz triangle theorem, see e.g.  \cite[Section 4.7]{H}, \cite[Section V.7]{Ne}. The basic point is that since the map $g_n$ maps the three connected components of $\R\setminus\{0,1\}$ to  arcs of circles in $\P^1$, it is given by a ratio of solutions to a hypergeometric equation, whose coefficients are determined by the angles at which these circles meet. In contrast to the usual setting of the Schwarz triangle theorem, the map $g_n$ is not injective on the upper half-plane, and one of the angles is $>\pi$. Although only small extensions to the usual argument are required to deal with this issue,   for the convenience of the reader we will give the argument in full. 

The main tool in the proof is the Schwarzian derivative, whose key properties  we briefly recall here. Let $f\colon U\to \bC$ be a holomorphic function on a domain $U\subset \bC$.
The Schwarzian derivative is defined by the expression
\[\cS f(t)=
    \bigg(\frac{f''(t)}{f'(t)}\bigg)'-\frac{1}{2} \bigg(\frac{f''(t)}{f'(t)}\bigg)^2.\]
   We shall need the following properties:
   \begin{itemize}
   \item[(S1)] Suppose that  $f'(t)$ is non-vanishing  on  a domain $U\subset \bC$. Thus
   $Q(t)=-\cS f(t)/2$ defines a holomorphic function $Q\colon U\to \C$. Then there exist two holomorphic solutions $y_1(t), y_2(t)$ to the differential equation
   \begin{equation*}
   y''(t)-Q(t)y(t)=0, \end{equation*}
such that $f(t)=y_1(t)/y_2(t)$.\smallskip

  \item[(S2)] If $R\in \operatorname{PGL}_2(\C)$ is a M{\"o}bius transformation, then $\cS(R\circ f)=\cS(f)$.\smallskip

  \item[(S3)] If $f\colon U\to \C$ and $g\colon V\to \C$ are composable holomorphic functions, then \[\cS(f\circ g)=(\cS(f)\circ g)\cdot (g')^2 + \cS(g).\]
  \end{itemize}


 We shall also need the following standard computation.

 \begin{lemma}
 \label{extraord}
 Suppose  that $k\colon U\to \C$ is a holomorphic function defined in a neighbourhood of $t_0\in U\subset \C$, which satisfies   $k(t_0)=0$ and $k'(t_0)\neq 0$. Take $\alpha\in \R$ and set $q(z)=z^\alpha$, viewed as a multi-valued  function on $\C^*$. Then    $\cS(q\circ k)$  is a single-valued  meromorphic function at $t=t_0$, and satisfies
\begin{equation} \label{extraord2}\cS(q\circ k)=\frac{1-\alpha^2}{2(t-t_0)^2}+O\Big(\frac{1}{t-t_0}\Big).\end{equation}
\end{lemma}

\begin{proof} Shrinking $U$ if necessary we can assume that the only zero of $k$ on $U$ is at $t=t_0$. Then $q\circ k$ is a multi-valued holomorphic function on $U\setminus\{t_0\}$, but the associated Schwarzian derivative $\cS(q\circ k)$ is clearly single-valued, since the various branches of $q$ differ by multiplication by a constant factor. A direct calculation shows that
\[\cS q(z)=\frac{1-\alpha^2}{2z^2}.\]
The result follows by applying property (S3) on subdomains of $U\setminus\{t_0\}$ on which $q\circ z$ is single-valued.
\end{proof}

\subsection{}
Consider the function $g_n(t)=\exp(\pi i f_n(t))$. By definition this is holomorphic  on the upper half-plane $\H\subset \C$, with non-vanishing derivative. Our first objective is to show that the Schwarzian derivative $\cS g_n(t)$ extends to a meromorphic function on the Riemann sphere $\P^1$, and to understand its leading order behaviour at the points $t=0,1,\infty$.

Take a point $t\in \R\setminus\{0,1\}$ and a neighbourhood $t\in B\subset \bC\setminus\{0,1\}$ invariant  under complex conjugation. By construction, the function $g_n$ is holomorphic on $B\cap\H$, with non-vanishing derivative. Moreover,  $g_n$ extends continuously over $B \cap \R$, and the extension is injective, and maps  $B \cap \R$ onto the arc of a circle $C\subset \P^1$. Let $R$ be a M{\"o}bius transformation taking this circle $C$ to the real axis $\R\subset \C$.  Then the composite $R\circ g_n$ is  real-valued on $B\cap \R$, so by the Schwarz reflection principle extends to a holomorphic function on $B$. It is easy to see that this function is locally univalent, and hence has non-vanishing derivative. It follows that   $\cS(R\circ g_n)$ extends to a holomorphic function on $B$ which is real-valued on $B \cap \R$. But by property (S2), this function coincides with $\cS g_n$. Thus we conclude that $\cS g_n$, which is   holomorphic on $\H\subset \C$ by construction, extends to a holomorphic function on a neighbourhood of each point of $\R\setminus\{0,1\}$, and is real-valued on the real axis. Applying  the reflection principle again we find that $\cS g_n$ extends to a holomorphic function on $\C\setminus \{0,1\}$. It thus remains to examine its behaviour at the points $\{0,1,\infty\}$.

Consider first the point $t=0$, and again take a small neighbourhood $0\in B\subset \bC$ invariant  under complex conjugation. As before, the function $g_n$ is holomorphic on $B\cap\H$, with non-vanishing derivative,  and extends continuously to $B\cap\R$. This time the extension maps the two halves of the real axis $B\cap \R_{\pm}$ into the arcs of two circles $C_\pm$ which meet at an angle $\pi/2$ at the point $g_n(0)=i^{2-n}$.  Let $R$ be a M{\"o}bius transformation which maps this point to $0$, and the circles $C_-$ and $C_+$ to the circles   $\R$ and  $\R\cdot i$ respectively.  The function $k(t)=R(g_n(t))^2$ is then holomorphic on $B\cap\H$, and continuous and real-valued on $B\cap\R$, so by the reflection principle extends to a holomorphic function $k\colon B\to \C$ satisfying $k(0)=0$. It is easy to see that since $g_n$ is locally univalent on $B\cap\H$, the same is true of $k$, and hence $k'(0)\neq 0$.
Applying Lemma \ref{extraord}  to the function $R(g_n(t))=k(t)^{1/2}$, and using property (S2), we conclude that $\cS g_n$ has a double pole at $t=0$, with leading order behaviour given by the right-hand side of \eqref{extraord2}, with $t_0=0$ and $\alpha=\tfrac{1}{2}$.


A very similar argument applies at the point  $t=\infty$. It is convenient to set $u=t^{-1}$ and consider the behaviour of the map $h_n(u)= g_n(u^{-1})$ near the point $u=0$. The above argument then applies, with the two circles $C_\pm$ now meeting at an angle $\pi/3$ at the point $h_n(0)=\omega$.  We consider a M{\"o}bius transformation $R$ which maps this point  to $0$, and the circles $C_-$ and $C_+$ to  the rays of  argument $0$ and $\pi/3$ respectively. We then consider the function $k(u)=R(h_n(u))^3$. We conclude that $\cS h_n$ has a double pole at $u=0$, with leading order behaviour given by \eqref{extraord} with $\alpha=\tfrac{1}{3}$. Using property (S3) it follows that the leading order behaviour of $\cS g_n$  near $t=\infty$ is
\[\cS g_n(t)=\frac{4}{9t^2}+O(t^{-3}).\]
In particular, $\cS g_n(t)$ has a double zero at $t=\infty$.

 Finally, we consider the point $t=1$.  We again take a small neighbourhood $1\in B\subset \bC$ invariant  under complex conjugation. A little care is required since the function $g_n$ is not injective on $B\cap \H$. Nonetheless, the two components of $B\cap \R\setminus\{1\}$   are mapped by $g_n$ to the two circles $r_-$ and $\R\cdot i^{2-n}$, which meet at an angle of $\tfrac{1}{2}(n-1)\pi$ at the point $g_n(1)=\infty$. We can therefore take a M{\"o}bius transformation $R$ which maps $\infty$ to $0$,  and maps these two arcs of circles to  the rays $\R\cdot i^{n-1}$ and $\R$.

 Introduce the functions
\[l(t)=R(g_n(t))^{1/(n-1)}, \qquad k(t)=l(t)^2,\]
defined by some fixed branch of $z\mapsto z^{1/(n-1)}$  near $0$.
Note that if we set $p(z)=z^{n-1}$ then we can write
\[l=(p^{-1}\circ R \circ p)\circ (p^{-1}\circ g_n).\]
Since $R(\infty)=0$, we have $R(z)=(cz+d)^{-1}$, for some $c,d\in \C$ with $c\neq 0$, and it is then easily checked that  $p^{-1}\circ R \circ p$ is  holomorphic  at $\infty$, with non-vanishing derivative. The region $R_n^-$ is closed under the map $z\mapsto z/(n-1)$, so the function $\exp(\pi i f_n(t)/(n-1))$,  which is a branch of $p^{-1}\circ g_n(t)$, is  holomorphic on $B\cap\H$, and  has a continuous extension to $B\cap\R$ sending $t=1$ to $z=\infty$. Since the region $R_n$ is contained in the strip $|\Re(z)|<\tfrac{1}{2}(n-1)$, this function is moreover injective. But note that $l$ cannot be extended to a holomorphic function on $B$ since it maps the two components of $B\cap \R\setminus\{1\}$ to circles meeting at an angle of $\pi/2$.

We have now shown that   $k(t)=l(t)^2$ is a holomorphic function on $B\cap\H$, which extends continuously to $B\cap\R$ and satisfies $k(1)=0$. By construction it is real-valued on $B\cap\R$, so by the reflection principle extends to a holomorphic function on $B$, and  the last remarks of the previous paragraph shows  that  $k'(1)\neq 0$.    Setting  $\alpha=(n-1)/2$ we can apply Lemma \ref{extraord} to the function $R(g_n(t))=k(t)^\alpha$ to conclude that $\cS g_n$ has a double pole at $t=1$, with leading order behaviour given by \eqref{extraord2}.

 \subsection{}
 Consider again the function $g_n(t)=\exp(\pi i f_n(t))$ on $\H$.  The  property (S1)  shows that $g_n(t)$ is given by the ratio $y_2(t)/y_1(t)$ of two solutions  to the differential equation
   \begin{equation}
   \label{labels}
   y''(t)-Q(t)y(t)=0, \qquad Q(t)=-\tfrac{1}{2}\cS g_n(t).\end{equation}
   The function $Q(t)$ extends to a meromorphic on $\P^1$, with double poles at $t=0,1$, no other poles,  and a double zero at infinity. It follows that it is uniquely determined by the leading terms in the Laurent expansions at $t=0,1,\infty$. Indeed, if two  such functions had the same  leading terms at these points, then their difference would have a triple zero at infinity, at worst simple poles at $t=0,1$, and no other poles, and any such function is zero.

  It remains to show that the solutions $y_i(t)$ can be written as  period integrals of the form \eqref{eq:twisted p}. We will show in the next subsection that these period integrals satisfy a hypergeometric equation
  \begin{equation}
\label{diff2}t(t-1)p''(t)+\big((a+b+1)t-c\big)p'(t) +ab\cdot p(t)=0,\end{equation}
with parameters
\[a=\tfrac{1}{2}\big(\tfrac{1}{3}-\nu\big), \qquad b=-\tfrac{1}{2}\big(\tfrac{1}{3}+\nu\big), \qquad c=\tfrac{1}{2}.\]
A standard calculation, e.g. \cite[Section 7]{Ne}, puts this in the $Q$-form \eqref{labels}, with $-2Q(t)$ having leading order behaviour at the points $t=0,1,\infty$ given by expressions of the form \eqref{extraord2}, suitably interpreted at $t=\infty$. The relevant angles are \[\alpha_0=1-c=\tfrac{1}{2}, \qquad \alpha_1=c-a-b=\tfrac{1}{2}+\nu, \qquad \alpha_\infty=b-a=\tfrac{1}{3}.\] Since these agree with what we computed above, it follows that the differential equation \eqref{labels} coincides with the $Q$-form of the equation \eqref{diff2}.

 The final thing to check is that we can take the paths $\gamma_i$ in the integral \eqref{eq:twisted p} to be integral cycles rather than complex linear combinations of such cycles.  Note that there is a unique solution to \eqref{diff2} up to scale which vanishes at $t=\infty$. Since $g_n(t)\to 0$ as $t\to \infty$ in $\H$ this solution must coincide with $y_2$. But this is indeed given by an integral cycle, namely the vanishing cycle. Since the other solution $y_1$ is obtained by applying monodromy transformations, it follows that it is also integral.  Indeed, at $t=0$ the monodromy has order 2 and preserves the image of the real axis under $g_n$, namely the unit circle. It follows that it is $z\mapsto \pm z^{-1}$.

\subsection{} In this section we prove that the twisted periods \eqref{eq:twisted p} satisfy the hypergeometric differential equation \eqref{diff2}. Let us consider $a\in \C$ to be fixed. As always we set $\nu=(n-2)/2$ with $n\geq 3$. Consider the function
\[f_a(h)=h^{-(\nu+1)}\int e^{h(u^3+au)} du=h^{-(\nu+\frac{4}{3})} \int e^{w^3+h^{2/3} a w} dw,\]
where we set $w=h^{1/3}\cdot  u$. Introduce the differential operator
\[D_h=h\partial_h +\nu+1.\]
Then
\[(D_h+\frac{1}{3}) f_a(h)= \frac{2}{3}\cdot h^{-(\nu+\frac{4}{3})} \int h^{\frac{2}{3}} \cdot aw\cdot e^{w^3+h^{2/3} a w} dw. \]
Repeating we obtain
\[(D_h-\frac{1}{3})(D_h+\frac{1}{3}) f_a(h) = \frac{4}{9}\cdot h^{-(\nu+\frac{4}{3})}\int h^{\frac{4}{3}} \cdot (aw)^2\cdot e^{w^3+h^{2/3} aw} dw,\]
and it follows that
\[\bigg((D_h-\frac{1}{3})(D_h+\frac{1}{3})+\frac{4a^3}{27}\cdot h^2\bigg) f_a(h)=h^{-\nu} \cdot \frac{4a^2}{27} \int   (3w^2+h^{\frac{2}{3}}a) \cdot e^{w^3+h^{2/3} aw} \, dw=0.\]

Now consider the inverse Laplace transform
\[p_a(b)=\int e^{bh} f_a(h) \,dh=\int \int e^{h(u^3+au+b)} \cdot h^{-(\nu+1)} \,du \,dh.\]
Exchanging the order of integration and using
\[\int e^{h(s+b)} \cdot h^{-(\nu+1)}dh= \Gamma(-\nu)\cdot (s+b)^{\nu},\]
where $\Gamma(x)$ denotes the gamma function, this becomes
\[p_a(b)=\Gamma(-\nu)\cdot \int (x^3+ax+b)^{\nu}\, dx.\]
Under the inverse transform $h \partial_h$ becomes $-b\partial_b -1$, so the transform of the operator $D_h$ is
$-b \partial_b +\nu$.
The twisted periods therefore satisfy the differential equation
\[\bigg((-b\partial_b +\nu+\frac{1}{3})(-b\partial_b+\nu-\frac{1}{3})+\frac{4 a^3}{27} \cdot \partial_b^2\bigg)\,p_a(b)=0.\]
Setting $t=-27b^2/4a^3$
this becomes
\begin{equation}
\label{diff}\Big(t(t-1)\partial_t^2+\Big((1-\nu)t-\frac{1}{2}\Big)\partial_t +\frac{1}{4}\big(\nu^2-\frac{1}{9}\big)\Big)p(t)=0,\end{equation}
which coincides with \eqref{diff2}.

\begin{figure}\centering
\begin{tikzpicture}[scale=.35]
\path (0,0) coordinate (O);
\path (4,0) coordinate (A);
\path (3+0.2,10) coordinate (m1);
\path (3+0.2,-10) coordinate (m2);
\draw[fill=gray!14] (m1)
    .. controls +(-90:3) and +(120:1) .. (A)
    .. controls +(-120:1) and +(90:3).. (m2)
   to [out=180,in=0] (-15,-10)
   .. controls +(90:3) and +(-90:1) ..  (-15,10) ;
\draw[white,thick] (m2) edge (-15,-10);
\draw[dotted,->,>=stealth] (-20,0) -- (8,0) node[right]{$x$};
\draw[dotted,->,>=stealth] (0,-10) -- (0,10) node[above]{$y$};
\draw (0,0) node[below left]{$0$};
\path[dotted] (6,-10) edge (6,10);
\path[dotted] (3,-10) edge (3,10);
\draw[red,thick] (A) edge (6,0);
\draw[red,thick] (m1) .. controls +(-90:3) and +(120:1) .. (A);
\draw[red,thick] (A)  .. controls +(-120:1) and +(90:3).. (m2);
\draw (6,0) node[below right] {$1$};
\path (5.5,0) node (a) {$$};
\path (3.6,1.5) node (b) {$$};
\path (3.6,-1.5) node (c) {$$};
\path[->,orange, bend right] (a) edge (b);
\path[->,orange, bend right] (b) edge (c);
\path[->,orange, bend right] (c) edge (a);
\path (4,6) node {$l_+$};\path (4,-6) node {$l_-$};
\draw[fill=black] (0,0) circle (.05);
\draw[fill=white] (6,0) circle (.1);
\draw[fill=red] (4,0) circle (.2);
%
%
\path (4,0) node[below right] {\tiny{$\frac{2}{3}$}};
\path (-15,0) node[left] {$-\infty$};
\draw[white,thick] (-15,10) edge (-15,-10);
\end{tikzpicture}
\caption{The region $R_\infty$}\label{fig:A2}
\end{figure}

\subsection{}
For the case $n=\infty$ we instead consider the region $R_\infty$ depicted in Figure \ref{fig:A2}. It is bounded by the same two curves $\ell_\pm$. We again consider the half region $R_\infty^-$ consisting of points of $R_\infty$ with negative imaginary part.
The boundary   has  two corners, namely $\tfrac{2}{3}$ and $\infty$.
Carath{\'e}odory's extension of the Riemann mapping theorem ensures that there is a  biholomorphism
\begin{equation}
\label{rm}f_\infty\colon \H\to R_\infty^+\end{equation}
which extends continuously over the boundary, and sends the points $(0,\infty)$ to $(\frac{2}{3},\infty)$. Considering the orientations of the two regions shows that $\R_{<0}$ is mapped to $\ell_-$, and $\R_{>0}$ to the open interval of the real axis $(-\infty,2/3)$.
Note that in this case the map $f_\infty(t)$ is not unique: precomposing it with a dilation of the form $t\mapsto \lambda\cdot t$ with $\lambda\in \R_{>0}$ gives another suitable map.

\begin{prop}
\label{osc}
One possible choice for the  function $f_\infty$ can be  written explicitly as
\[f_\infty(t)= \frac{1}{\pi i} \log\bigg(\frac{\phi^{(2)}_\infty(a,b)}{\phi^{(1)}_\infty(a,b)}\bigg),\]
where $t=a^3$, $b$ is arbitrary, and
\begin{equation}\label{osci}\phi_\infty^{(i)}(a,b)=\int_{\delta_i} \exp\big(x^3+ax+b\big) \, dx,\end{equation}
for  certain  contours $\delta_i\colon \R\to \bC$ which satisfy $\Re \delta(x)^3\to -\infty$ as $x\to \pm \infty$.
\end{prop}

Note that the function $f_\infty$ only depends on $a$, because changing $b$  multiplies both functions $\phi_\infty^{(i)}(a,b)$ by an equal factor,  and hence leaves their ratio unchanged.

\begin{proof}
We shall give a direct proof in this case. Consider the Airy function
\[\Ai(a)=\frac{1}{2\pi i}\int_\delta \exp\Big(\frac{x^3}{3}-ax\Big) dx,\]
where the path of integration $\delta\colon \R\to \bC$ tends to $\infty$ along the rays of argument $\pm \pi/3$ as $x\to \pm \infty$.  We shall need two standard properties of this function, see e.g. \cite[Section 8.9]{H} or \cite[Section 11.8]{O}. Firstly, it is an entire function of $a\in \C$, with zeroes only on the negative real axis. Secondly, there is an asymptotic expansion of $\Ai(a)$ as $a\to \infty$ which implies that there is a constant $R\in \bC$ such that
\begin{equation}
\label{asy}\Ai(a)\cdot \exp\bigg(\frac{2a^{\frac{3}{2}}}{3} \bigg)\cdot a^{\frac{1}{4}}\to R,\end{equation}
as $a\to \infty$ in any closed subsector of $\bC\setminus\R_{<0}$. In particular, it follows  that $\Ai(a)\to 0$ as $a\to \infty$ in the closed sector $|\arg(a)|\leq \pi/3$.

Let us introduce the functions
\[g(a)=\frac{\omega^2 \Ai(-\omega^2 a)}{\omega \Ai(-\omega a)}, \qquad h(a)=\frac{1}{\pi i}\log g(a), \qquad f(t)=h\big(t^{1/3}\big).\]
By what was said above, the function  $g(a)$ is meromorphic, with zeroes and poles only on the rays $\R_{>0}\cdot \omega^{\pm 1}$. In particular, it is regular and non-vanishing on a neighbourhood of the closed sector $\Sigma\subset \bC$ bounded by the rays of argument $0$ and $\pi/3$. Choosing the principal branch of $\log$, the function $h(a)$ is then  well-defined and holomorphic on this sector, and satisfies $h(0)=2/3$. To define the function $f(t)$ we choose the branch of the cube root function on $\H$ which lies in $\Sigma$.   Then $f(t)$ is also well-defined and holomorphic on $\H$.

We will show that the function $f(t)$ gives a possible choice for the  map \eqref{rm}. In view of the obvious relation
\begin{equation}
\label{airy}\omega^i\cdot \Ai(\omega^i a)=\frac{1}{2\pi i}\int_{\omega^i\cdot \delta} \exp\Big(\frac{x^3}{3}-ax\Big) dx,\end{equation}
this will be enough to prove the result. Note that the factor of $\tfrac{1}{3}$ multiplying $x^3$ can be absorbed by rescaling $t$ by a  factor of $3$, which is a transformation of the form mentioned just before the statement of the Proposition.

It is an immediate consequence of \eqref{airy} that
\[\Ai(a)+\omega \Ai(\omega a)+\omega^2\Ai(\omega^2a)=0,\]
and it follows that
\begin{equation}
\label{aaaa}g(\omega^i a)=T^{i} (g(a)),\end{equation}
where $T(z)=-(z+1)/z$ is the M{\"o}bius transformation  of order 3 appearing in \eqref{drn}.
The definition of the Airy function implies that
$\overline{\Ai( a)}=\Ai(\bar{a}),$ and hence that
\begin{equation}
\label{bbbb}\overline{g( a)}=g( \bar{a})^{-1}.\end{equation}
It follows that $g$ maps points on the real axis onto  the unit circle $C_0$. Using \eqref{drn} and \eqref{aaaa} this implies that $g$ maps the ray of argument $\pi/3$, which is part of the circle $\R\cdot \omega^2$, to the circle $C_-=T^{-1}(C_0)$.

Putting all this together we find that the function $f$ is holomorphic on the upper half-plane and extends continuously to the boundary. It satisfies $f(0)=\tfrac{2}{3}$ and maps the boundary rays $\R_{<0}$ and $\R_{>0}$ to the images under $(1/i\pi)\log (z)$ of the circles $C_-$ and $C_0$ respectively. The asymptotic property \eqref{asy} shows that $f$  extends continuously over $\infty$.

Let us now compose $f$ with the inverse of the Riemann map  \eqref{rm}. The resulting function $k=f\circ f_\infty^{-1}$ defines a holomorphic map $k\colon \H\to \H$ which extends continuously over the boundary, and satisfies $k(\R)\subset \R$. By the  reflection principle, $k$  extends to a meromorphic function on $\P^1$, which has  a single pole at $\infty$. The asymptotic formula \eqref{asy} shows that $f(t)$ behaves like a constant multiple of $t^{1/2}$ as $t\to \infty$. It follows from this that  the pole of $k$ is simple, and hence that $k(t)=\lambda\cdot t$ for some $\lambda\in \R_{>0}$.
\end{proof}


\section{Stability conditions}
\label{curry}
In this section we analyse the space of stability conditions on the categories $\D_n$ and give the proof of the main result, Theorem \ref{first}. We refer to \cite{B} and \cite[Section 7]{BrSm} for basic definitions concerning stability conditions. The basic idea is to identify a fundamental domain $\cU_n$ for the action of the group $\P \Aut_*(\D_n)$ on the space of projective stability conditions $\P\Stab_*(\D_n)$, and show that its image under the logarithmic central charge map can be identified with the  region $R_n$ of Section \ref{conf}.


\subsection{} \label{bking}Take $3\leq n\leq \infty$ and let $\Stab(D_n)$ denote the space of stability conditions on the triangulated category $\D_n$. In particular, all stability conditions in this space are assumed to satisfy the support property. There is a distinguished connected component  $\Stab_*(\D_n)\subset \Stab(D_n)$ which  contains all stability conditions $\sigma=(Z,\mathcal{P})$ whose heart $\mathcal{P}((0,1])$ coincides with the canonical heart $\A_n\subset \D_n$.

 The group of autoequivalences $\Aut(\D_n)$ acts on $\Stab(\D_n)$ by
\[\Phi\cdot (Z,\mathcal{P})=(Z',\mathcal{P}'), \qquad Z'(E)= Z(\Phi^{-1}(E)), \qquad \mathcal{P}'(\phi)=\Phi(\mathcal{P}(\phi)).\]
Recall that  an autoequivalence $\Phi\in \Aut(\D_n)$ is negligible if it fixes the objects $S_1$, $S_2$. Any such autoequivalence  acts trivially on an open subset of $\Stab_*(\D_n)$, and hence on the whole connected component $\Stab_*(\D_n)$. Conversely, any autoequivalence which acts trivially on $\Stab_*(\D_n)$ must be negligible.

There is an action of $\C$ on $\Stab(\D_n)$, commuting with the action of $\Aut(\D_n)$, given by the rule
\[\lambda\cdot (Z,\mathcal{P})=(Z',\mathcal{P}'), \qquad Z'(E)=\exp(-i\pi \lambda)\cdot Z(E), \qquad \mathcal{P}'(\phi)=\mathcal{P}(\phi+\Re(\lambda)).\]
Note that the action of the shift $[k]\in \Aut(\D_n)$  coincides with the action of $k\in \C$.
This action of $\C$ on $\Stab(\D_n)$ is free and proper, and we can consider the quotient complex manifold
\[\P\Stab(\D_n)=\Stab(\D_n)/\C,\]
together with the distinguished connected component $\P\Stab_*(\D_n)\subset \P\Stab(\D_n)$.
The central charge map induces a local isomorphism of complex manifolds \begin{equation}
\label{4am}\P\Stab_*(\D_n)\to \P \Hom_\Z(K_0(\D_n),\C)\isom \P^1.\end{equation}

We refer to points of the quotient space $\P\Stab_*(\D_n)$ as \emph{projective stability conditions}. Note that a projective stability condition determines full subcategories of stable and semistable objects of $\D_n$, and  a well-defined   phase difference $\phi(B)-\phi(A)\in \R$ for any two semistable objects $A,B\in \D_n$.


\subsection{}
Let us again take $3\leq n\leq \infty$ and define the following subsets of the space of projective stability conditions $\P\Stab_*(\D_n)$.

\begin{defn} \label{fundomain}
\begin{itemize}
\item[(a)] Let $\cU_n^+\subset \P\Stab_*(\D_n)$ be the subset of projective stability conditions for which  the objects $S_1$ and $S_2$ are both stable and \begin{equation}
\label{inp}0\leq  \phi(S_2)-\phi(S_1)< (n-2)/2.\end{equation}
\item[(b)] Let $\cU_n^-\subset \P\Stab_*(\D_n)$ be the subset of projective stability conditions for which the objects $S_1$, $S_2$ and $E$ are all stable, and  both of the following  inequalities  hold:
\begin{equation}
\label{in}\phi(S_1)-\phi(S_2) < \phi(S_2[1])-\phi(E),\quad  \phi(S_1)-\phi(S_2)< \phi(E)-\phi(S_1[-1]).
\end{equation}
\item[(c)] Let $\cU_n=\cU_n^-\cup\cU_n^+\subset \P\Stab_*(\D_n)$.
\end{itemize}
\end{defn}

In Prop. \ref{fundom} below we will show that the subset $\cU_n$ is a fundamental domain for the action of $\P\Aut_*(\D_n)$. The following result is a first step towards this.

\begin{lemma}
\label{lem}
For any $3\leq n\leq \infty$
the subset $\cU_n\subset \P\Stab_*(\D_n)$ is open. Moreover, if a
projective stability condition lies in the boundary of $\cU_n$ then one of the following two conditions is satisfied:
\begin{itemize}
\item[(i)] the set of stable objects up to shift is $\{S_1,S_2\}$, and \begin{equation}\label{in1}\phi(S_2)-\phi(S_1)= (n-2)/2;\end{equation}
\item[(ii)] the set of stable objects up to shift is $\{S_1,S_2,E\}$, and the inequalities
\begin{equation}
\label{in2}\phi(S_1)-\phi(S_2)\leq  \phi(S_2[1])-\phi(E),\quad \phi(S_1)-\phi(S_2)\leq  \phi(E)-\phi(S_1[-1])\end{equation}
hold, with one or both being an equality.
\end{itemize}
In particular, when $n=\infty$, only case (ii) can occur.
\end{lemma}

\begin{proof}
It is a general fact \cite[Prop. 7.6]{BrSm} that if an object $S$ is stable with respect to some stability condition then the same is true for all stability conditions in an open neighbourhood. Moreover the phase $\phi(S)$ varies continuously in this neighbourhood. It follows immediately that if  $\bar{\sigma}\in \cU_n$ satisfies $\phi(S_1)\neq \phi(S_2)$, then an open neighbourhood of $\bar{\sigma}$ also lies in $\cU_n$. On the other hand, any projective stability condition  for which $S_1$, $S_2$ are semistable with $\phi(S_1)=\phi(S_2)$ can be lifted to a   stability condition  with standard heart $\A_n$. The object $E$ is then necessarily semistable with $\phi(E)=\phi(S_i)$, and all points in an open neighbourhood of $\bar{\sigma}$  lie in either $\cU_n^-$ or $\cU_n^+$ depending on the sign of $\phi(S_2)-\phi(S_1)$. Thus $\cU_n$ is open.

Consider  $\bar{\sigma}\in\P\Stab_*(\D_n)$  lying in the closure of $\cU_n^+$. Then the objects $S_1$ and $S_2$ are semistable, and the non-strict version of the inequality \eqref{inp} holds. In particular $n<\infty$. It follows that $\bar{\sigma}$ can be lifted to a stability condition whose heart contains, and hence is equal to, the heart  $\A_n(k)=\<S_1[k],S_2\>$ of Prop. \ref{tilts}, for some $0\leq k\leq (n-2)/2$. This implies that the objects $S_1$ and $S_2$ are in fact stable. If $\bar{\sigma}$ lies in the boundary of $\cU_n^+$, so that $\bar{\sigma}\notin \cU_n^+$, it follows that the equality \eqref{in1} must hold.

Suppose now that $\bar{\sigma}\in\P\Stab_*(\D_n)$  lies in the closure of $\cU_n^-$.  Then the objects $S_1$, $S_2$ and $E$ are all semistable, and the  the inequalities \eqref{in2} both hold. The triangle associated to the short exact sequence
\eqref{ses} shows that for any lift \[\phi(S_2)\leq \phi(E)\leq \phi(S_1)\leq \phi(S_2)+1.\]  We cannot have $\phi(S_1)=\phi(S_2)+1$ since then either $\phi(E)=\phi(S_1)$, or $\phi(E)=\phi(S_2)$, and in either case one of the equalities \eqref{in2} fails. Thus $\phi(S_1)-\phi(S_2)<1$, and hence $\bar{\sigma}$ can be lifted to a stability condition with heart $\A_n$. If $\bar{\sigma}$ lies in the boundary of $\cU_n$ we cannot have $\phi(S_1)=\phi(S_2)$ since as above, such stability conditions lie in $\cU_n$. It follows that $S_1$, $S_2$ and $E$ are all in fact stable, and again, since $\bar{\sigma}\notin \cU_n$, at least one of the inequalities \eqref{in2} must be an equality.
\end{proof}

\subsection{}
The following result gives the link with the regions considered in Section \ref{conf}.

\begin{prop}
\label{rain}
Take $3\leq n \leq \infty$. Then the function \[g_n(\bar{\sigma})=\frac{1}{\pi i} \log \Big(\frac{Z(S_1)}{Z(S_2)}\Big)\]
defines a biholomorphic map between the regions $\cU_n$ and $R_n$, where the branch of $\log$ is chosen so that $\Re g_n(\sigma)=\phi(S_1)-\phi(S_2)$.
\end{prop}

\begin{proof}
It follows from the  proof of Lemma \ref{lem}  that if $\bar{\sigma}\in \cU_n$ then the set of stable objects up to shift is precisely $\{S_1,S_2\}$ or $\{S_1,S_2,E\}$ depending on whether $\bar{\sigma}\in \cU_n^+$ or $\bar{\sigma}\in \cU_n^-$. Indeed, if $\bar{\sigma}\in \cU_n^+$ then  it can be lifted to a stability condition whose heart  is $\A_n(k)=\<S_1[k],S_2\>$  for some $0\leq k\leq (n-2)/2$. If $k>0$ this heart is semisimple, so the only stable objects are $S_1$, $S_2$ up to shift, and if $k=0$ the same result holds since $\phi(S_1)\leq \phi(S_2)$. On the other hand, if $\bar{\sigma}\in \cU_n^+$ then $0<\phi(S_1)-\phi(S_2)<1$, and $\bar{\sigma}$ can be lifted to a stability condition with heart $\A_n$, which implies that the stable objects are precisely $S_1$, $S_2$ and $E$.

Let us consider the two halves $\cU_n^\pm$ of the fundamental domain  separately. The region $\cU_n^+$ consists of projective stability conditions $\bar{\sigma}$ for which   the only stable objects are $S_1$ and $S_2$ up to shift. It is easy to see that any such $\bar{\sigma}$   is  determined  by the value of $z=g_n(\bar{\sigma})\in \C$, and using the finite-length hearts $\A_n(k)=\<S_1[k],S_2\>$  of Prop. \ref{tilts} it is also easy to see that all  values of $z\in \C$ compatible with the constraint \eqref{inp} are realised. Thus $g_n$ maps the region $\cU^+_n$  bijectively onto the region $(2-n)/2<\Re(z)\leq 0$.

We showed above that for all projective stability conditions $\bar{\sigma}$ in the region $\cU^-_n$, the objects $S_1,S_2$ are stable,  and $0<\phi(S_1)-\phi(S_2)<1$. It is easy to see that the set of such projective stability conditions is mapped  bijectively by $g_n$ onto the strip $0<\Re(z)<1$. Since $Z(E)=Z(S_1)+Z(S_2)$, an elementary piece of geometry shows that
\[\phi(E)-\phi(S_2)<\phi(S_1)-\phi(E)\iff |Z(S_2)|<|Z(S_1)|.\]
Recalling from Lemma \ref{little} and Prop. \ref{easy} the action of $\Sigma$ on objects, and noting that by \eqref{5am} it induces the M{\"o}bius transformation $T$ of Section \ref{conf} on the projective space \eqref{4am}, we find that the inequalities \eqref{in} can be alternatively expressed as
\[|Z(S_1)+Z(S_2)|>|Z(S_2)|, \qquad |Z(S_1)+Z(S_2)|>|Z(S_1)|.\]
In terms of $z=Z(S_2)/Z(S_1) $ this gives the inequalities \eqref{eq:ineq} defining the grey area of Figure ~\ref{fig:circles}.
Thus the image of $\cU^-_n$ is the part of the strip $0<\Re(z)<1$ bounded by the images of the circles $C_\pm$ under the map $(1/i\pi) \log (z)$. \end{proof}

\subsection{}
The following  result will be a key ingredient in our proof of Theorem \ref{first}.

\begin{prop}
\label{fundom}
Take $3\leq n\leq \infty$. Then
\begin{itemize}
\item[(i)] an autoequivalence $\Phi\in \Aut(\D_n)$ is reachable precisely if its action on $\Stab(\D_n)$ preserves the connected component  $\Stab_*(\D_n)\subset \Stab(\D_n)$;\smallskip

\item[(ii)] the open subset $\cU_n\subset \P\Stab_*(\D_n)$ is a fundamental domain for the action of $\P\Auts(\D_n)$.
\end{itemize}
\end{prop}

\begin{proof}
Suppose   that $\bar{\sigma}$ lies in the intersection $\cU_n\cap \Phi^{-1}(\cU_n)$ for some $\Phi\in \Aut(\D_n)$, so  that $\bar{\sigma}\in\cU_n$ and also $\Phi(\bar{\sigma})\in \cU_n$.
Consider the case that $\bar{\sigma}\in \cU^+_n$. Then $\Phi$ preserves the  set of stable objects $\{S_1,S_2\}$ up to shift.  Suppose that $\Phi$  exchanges the two objects up to shift. Then given the degrees of the maps between $S_1$ and $S_2$ it follows that $n<\infty$. Using Lemma \ref{little} we then obtain
\[\Phi(S_1,S_2)=\Upsilon(S_1,S_2)=(S_2,S_1[n-2]))\]
up to shift, and the inequality \eqref{inp} then gives a contradiction. Thus $\Phi$ preserves the objects $S_i$ up to a shift, and it follows easily that $\Phi[d]$ is negligible for some $d\in \Z$, and hence that $\Phi$ is reachable and defines the  identity element of $\P\Aut_*(\D_n)$.

Consider now the  case when $\bar{\sigma}\in \cU^-_n$. Then $\Phi$ preserves the set of stable objects $\{S_1,S_2,E\}$ up to shift. Suppose that $\Phi$ defines a non-trivial  permutation of these objects  up to shift.
From Lemma \ref{little} and Prop. \ref{easy} we know that
\begin{equation}\label{yo}\Sigma(S_1,E,S_2)=(S_2[1],S_1,E),\qquad \Sigma^{-1}(S_1,E,S_2)=(E,S_2,S_1[-1]),\end{equation}
and from the degrees of the maps between $S_1$, $S_2$ and $E$ it follows
 that $\Phi=\Sigma^{\pm 1}$. Comparing \eqref{yo} with the form of the inequalities \eqref{in} gives a contradiction. Thus again $\Phi$ preserves the objects $S_i$ up to a shift, and hence defines the  identity element of $\P\Aut_*(\D_n)$.

Consider the union
\begin{equation}
\label{uss}W=\bigcup_{\Phi\in \P\Auts(\D_n)} \Phi(\bar{\cU}_n)\subset \P\Stab_*(\D_n).\end{equation} By Lemma \ref{lem}, for any projective stability condition $\bar{\sigma}$ lying  in the closure of $\cU_n$, the set of stable objects  up to shift is contained in $\{S_1,S_2,E\}$, and hence $\bar{\sigma}$ can only lie in the closure of a different region $\Phi(\cU_n)$ for finitely many  $\Phi\in \P\Auts(\D_n)$. Thus the union in \eqref{uss} is locally-finite, and hence $W$ is a closed subset.

We now prove that $W$ is an open subset. Consider  $\bar{\sigma}$ lying in the boundary of $\cU_n$, and a small open  neighbourhood $\bar{\sigma}\in V\subset \P\Stab_*(\D_n)$.  According to Lemma \ref{lem} there are two cases. In case (i) the stable objects are $S_1$, $S_2$ up to shift, and if $V$ is small enough these remain stable, and all points of $V$ lie in the closure of either $\cU_n$ or $\Upsilon(\cU_n)$ depending on which of the positive numbers
\[\phi(S_2)-\phi(S_1), \qquad \phi(S_1[n-2])-\phi(S_2),\]
is the smallest. In case (ii) the stable objects are $S_1,S_2$ and $E$. Once again, if $V$ is small enough these objects remain stable, and  all points of $V$ lie in the closure of one of the regions $\Sigma^i(\cU_n)$ depending on which of the positive numbers
\[\phi(S_1)-\phi(S_2),\qquad  \phi(S_2[1])-\phi(E),\qquad \phi(E)-\phi(S_1[-1]),\]
is the smallest. This completes the proof that $W$ is open. Since, by Prop. \ref{bug}(a), the autoequivalences $\Phi$ and $\Sigma$ generate $\P\Aut_*(\D_n)$, we also see that $W$ is connected. Thus  $W=\P\Stab_*(\D_n)$, and
in particular, any reachable autoequivalence preserves the connected component $\Stab_*(\D_n)$.

We have now proved one half of part (i) and part (ii). To complete the proof of (i)  take an autoequivalence $\Psi\in \Aut(\D_n)$ preserving the distinguished  component $\Stab_*(\D_n)$. Then $\Psi$ maps $\cU_n$ into  the union \eqref{uss}. Thus we can find a point $\bar{\sigma}\in \cU_n$, and a reachable autoequivalence $\Phi\in \Aut_*(\D_n)$ such that $\Psi(\bar{\sigma})=\Phi(\bar{\sigma})$. Applying the first two paragraphs of this proof to the composite $\Psi\circ\Phi^{-1}$  it follows that $\Psi$  is reachable.
\end{proof}

\begin{remark}
\label{forw}
Suppose a projective stability condition $\bar{\sigma}\in \P\Stab_*(\D_n)$ is fixed by some non-trivial element $\Phi\in \Aut_*(\D_n)$. By Prop. \ref{fundom}  it is enough to consider the case  when $\bar{\sigma}$ lies in the closure of the fundamental domain $\cU_n$. The proof of Prop. \ref{fundom} then shows that  there are two possibilities:
\begin{itemize}
\item[(a)]  the stable objects are $S_1,S_2$ up to shift, $\Phi=\Upsilon$, and
\[Z(S_2)=i^{n-2}\cdot Z(S_1), \qquad \phi(S_2)-\phi(S_1)=\tfrac{1}{2}(n-2);\]
\item[(b)]  the stable objects are $S_1,S_2$ and $E$ up to shift,  $\Phi=\Sigma^{\pm 1}$, and
\[Z(E[1])=\omega\cdot Z(S_1)=\omega^2\cdot Z(S_2), \]\[\phi(S_1)-\phi(S_2)=\phi(S_2[1])-\phi(E)= \phi(E)-\phi(S_1[-1])=\frac{2}{3}.\]
\end{itemize}

These two cases are illustrated in Figure \ref{special}.
For $n<\infty$  it is easy to see that both possibilities occur, and that the quotient
\[\P\Stab_*(\D_n)/\P\Auts(\D_n),\]
therefore has two orbifold points, with stabilizer groups  $\mu_2$ and $\mu_3$ respectively.
In the case $n=\infty$ only the second case occurs, and we obtain a single $\mu_3$ orbifold point in the quotient.
\end{remark}

\begin{figure}
\begin{center}
\begin{tikzpicture}[scale=1.5]
 \draw [->,>=stealth,dotted] (-1.5,0) -- (1.5,0);
 \draw [->,>=stealth,dotted] (0,-0.1) -- (0,1.5);
 \draw [->,thick,>=stealth](0,0) -- (0,1) node [anchor=east]{$S_2$};
 \draw [->,thick,>=stealth] (0,0) -- (1,0) node [anchor = north]{$S_1$};
 \draw [->,>=stealth,dotted] (2.5,0) -- (5.5,0);
 \draw [->,>=stealth,dotted] (4,-0.1) -- (4,1.5);
 \draw [->,thick,>=stealth](4,0) -- (4.5,0.86) node [anchor=west]{$E$};
 \draw [->,thick,>=stealth](4,0) -- (3.5,0.86) node [anchor=east]{$S_1$};
 \draw [->,thick,>=stealth] (4,0) -- (5,0) node [anchor = north]{$S_2$};
\end{tikzpicture}
\end{center}
\caption {Orbifold points in the boundary of $\cU_n$  when  $n$ is odd.\label{special}}
\end{figure}

\begin{figure}[t]\centering
\begin{tikzpicture}[scale=.7]
\path (0,0) coordinate (O);
\path (0,6) coordinate (S1);
\draw[fill=\graphc!7,dotted]
    (S1) arc(360/3-360/3+180:360-360/3:10.3923cm) -- (O) -- cycle;
\draw[,dotted] (210:6cm) -- (O);
\draw[fill=black] (30:1.6077cm) circle (.05cm);
\draw
    (O)[Periwinkle,dotted,thick]   circle (6cm);
\draw
    (S1)[\graphc,thick]
    \foreach \j in {1,...,3}{arc(360/3-\j*360/3+180:360-\j*360/3:10.3923cm)} -- cycle;
\draw
    (S1)[\graphc,semithick]
    \foreach \j in {1,...,6}{arc(360/6-\j*360/6+180:360-\j*360/6:3.4641cm)} -- cycle;
\draw
    (S1)[\graphc]
    \foreach \j in {1,...,12}{arc(360/12-\j*360/12+180:360-\j*360/12:1.6077cm)}
        -- cycle;
\foreach \j in {1,...,3}
{\path (-90+120*\j:.7cm) node[\vertexc] (v\j) {$\bullet$};
 \path (-210+120*\j:.7cm) node[\vertexc] (w\j) {$\bullet$};
 \path (-90+120*\j:2.2cm) node[\vertexc] (a\j) {$\bullet$};
 \path (-90+15+120*\j:3cm) node[\vertexc] (b\j) {$\bullet$};
 \path (-90-15+120*\j:3cm) node[\vertexc] (c\j) {$\bullet$};}
\foreach \j in {1,...,3}
{\path[->,>=stealth] (v\j) edge[\thirdc,bend left,thick] (w\j);
 \path[->,>=stealth] (a\j) edge[\thirdc,bend left,thick] (b\j);
 \path[->,>=stealth] (b\j) edge[\thirdc,bend left,thick] (c\j);
 \path[->,>=stealth] (c\j) edge[\thirdc,bend left,thick] (a\j);}
\foreach \j in {1,...,3}
{\path (60*\j*2-1-60:3.9cm) node[\vertexc] (x1\j) {$\bullet$};
 \path (60*\j*2+1-120:3.9cm) node[\vertexc] (x2\j) {$\bullet$};
}
\foreach \j in {1,...,3}
{\path[->,>=stealth] (v\j) edge[\halfc,bend left,thick] (a\j);
 \path[->,>=stealth] (a\j) edge[\halfc,bend left,thick] (v\j);
 \path[->,>=stealth] (b\j) edge[\halfc,bend left,thick] (x1\j);
 \path[->,>=stealth] (c\j) edge[\halfc,bend left,thick] (x2\j);
 \path[->,>=stealth] (x1\j) edge[\halfc,bend left,thick] (b\j);
 \path[->,>=stealth] (x2\j) edge[\halfc,bend left,thick] (c\j);}
\end{tikzpicture}
\caption{Exchange graph as the skeleton of space of stability conditions}\label{fig:LOGO}
\end{figure}

\subsection{}
We now  prove  a projectivised version of our main result, Theorem \ref{first}.
Recall from Section \ref{sec:intro}  that the quotient $\h/W$ is isomorphic to $\C^2$, and has  co-ordinates $(a,b)$ obtained by writing
\[p(x)=(x-u_1)(x-u_2)(x-u_3)=x^3+ax+b.\]
The image of the root hyperplanes $u_i=u_j$ is the discriminant
\[\Delta=\{(a,b)\in \C^2: 4a^3+27b^2=0\}.\]

The space $\h$ has a  $\C^*$ action which rescales the  co-ordinates $u_i$ with weight 1.
This acts on $(a,b)$ with weights $(2,3)$. We thus have
\[\C^*\backslash (\h\setminus\{0\})/W\isom \P(2,3).\]
The weighted projective space $\P(2,3)$  contains two orbifold points which we label by their stabilizer groups: thus $\mu_2=[1:0]$ and $\mu_3=[0:1]$. The image in $\P(2,3)$ of the discriminant is a single non-orbifold point  $\Delta=[-3:2]$. The coarse moduli space map is
\[\P(2,3)\to \P^1, \qquad [a:b]\mapsto [a^3:b^2].\]
We take the affine co-ordinate  $t=-27b^2/4a^3$ on the coarse moduli space. The points  $(\mu_2,\Delta, \mu_3)$ then correspond to $t=(0,1,\infty)$ respectively.

In the case $n=\infty$  we consider the  action of $\mu_3$ on  $\C$ given by $a\mapsto \omega\cdot a$. The quotient has a single orbifold point at $a=0$ with stabilizer group $\mu_3$. We take the co-ordinate $t=a^3$ on the coarse moduli space, which is isomorphic to $\C$.

\begin{thm}
\label{proj}
\begin{itemize}
\item[(a)] For $3\leq n<\infty$ the action of $\P\Auts(\D_n)$ on $\P\Stab_*(\D_n)$ is proper and quasi-free and there is an isomorphism of complex orbifolds
\[\P\Stab_*(\D_n)/\P\Auts(\D_n)\isom \P(2,3)\setminus\Delta.\]

\item[(b)] The action of $\P\Aut(\D_\infty)$ on $\P\Stab(\D_\infty)$ is proper and quasi-free and there is an isomorphism of complex orbifolds
\[\P\Stab(\D_\infty)/\P\Auts(\D_\infty)\isom \C/\mu_3.\]
\end{itemize}
\end{thm}

\begin{proof}
Consider part (a) first. We shall actually prove more, namely that there is an isomorphism of complex manifolds
\begin{equation}
\label{desp}h_n\colon \widetilde{\P(2,3)\setminus\Delta}\lra \P\Stab_*(\D_n),\end{equation}
where the space on the left is the orbifold universal cover, with the following two properties:
\begin{itemize}
\item[(i)]
 it intertwines the action of the orbifold fundamental group   on the left, with the action of $\P\Auts(\D_n)$ on the right, under the homomorphism
 \begin{equation}
 \label{row}\rho\colon \pi_1 (\P(2,3)\setminus \Delta)\to \P\Auts(\D_n),\end{equation}
 which sends the local generators at the  orbifold points $\mu_2$ and $\mu_3$ to the autoequivalences $\Upsilon$ and $\Sigma$ respectively;
\item[(ii)] it makes the following diagram commute
\[\begin{tikzpicture}
\draw(0,1) node (A) {$\widetilde{\P(2,3)\setminus\Delta}$}
     (5,1) node (B) {$\P\Stab_*(\D_n)$}
     (0,-1) node (C) {$\P^1$}
     (5,-1) node (D) {$\P\Hom_{\Z}(K_0(\D_n),\bC)$};
\draw[->,>=stealth]
    (A) -- (B) node[above, midway]{$h_n\quad$};
\draw[->,>=stealth]
    (A) edge (C)
    (B) edge (D) ;
\draw[->,>=stealth]
    (C) -- (D) node[above, midway]{$\quad\cong$};
\end{tikzpicture}\]
in which the vertical arrows are local analytic isomorphisms, given on the left by a particular ratio of twisted periods \eqref{eq:twisted p}, and on the right by the central charge map.
\end{itemize}

The fact that conditions (i) and (ii) are compatible is the statement that the monodromy representation of the equation \eqref{diff2} agrees with the action of the map $\P\Aut_*(\D_n)$ on central charges under the homomorphism $\rho$. This holds because the M\"{o}bius maps specified by the mondromy of the hypergeometric differential equation are precisely those defined by the matrices given in (\ref{5am})

 To construct a suitable map \eqref{desp}, first consider the subset $P_n\subset \P^1$ which is the union of the lower and upper half-planes $\pm \Im(t)>0$ in the coarse moduli space of $\P(2,3)$, glued along the segment
$(0,\infty)=(\mu_2,\mu_3)$.  We can glue  the isomorphism $f_n\colon \H\to R_n^-$   of  Section \ref{conf} to its complex conjugate  to obtain an isomorphism $f_n\colon P_n\to R_n$. Composing
this with the inverse of the  isomorphism $g_n\colon \cU_n\to R_n$ of Prop. \ref{rain}  gives a biholomorphic map
$h_n=g_n^{-1}\circ f_n\colon P_n\to \cU_n$. By construction this map satisfies the condition (ii).

Since $P_n$ is simply-connected we can lift it to the universal cover and hence view it as  a subset of the space appearing on the left of \eqref{desp}. The subsets
$P_n$ and $\cU_n$   are then both fundamental domains for the relevant group actions, so there is a unique way to extend $h_n$ uniquely so as to satisfy condition (i) on the dense  open subset which is the disjoint union of the translates of the fundamental domains. Since the conditions (i) and (ii) are compatible, the resulting map $h_n$ also satisfies (ii). Using the local homeomorphisms to $\P^1$ we can then extend $h_n$  over the boundaries of the fundamental domains to obtain the required isomorphism.

The proof of part (b) proceeds along similar lines.  The universal cover of the orbifold $\C/\mu_3$ is  just $\C$ itself, so in this case we look for an isomorphism
\[h_\infty\colon \C\to  \P\Stab_*(\D_\infty),\]
satisfying the properties
\begin{itemize}
\item[(i)]
it intertwines the action of $\mu_3$ on both sides, given on the left by $a\mapsto \omega\cdot a$, and on the right by the autoequivalence $\Sigma$;
\item[(ii)]it makes the following diagram commute
\[\begin{tikzpicture}
\draw(0,1) node (A) {$\C$}
     (5,1) node (B) {$\P\Stab_*(\D_\infty)$}
     (0,-1) node (C) {$\C/\mu_3$}
     (5,-1) node (D) {$\P\Hom_{\Z}(K_0(\D_\infty),\bC)$};
\draw[->,>=stealth]
    (A) -- (B) node[above, midway]{$h_\infty\quad$};
\draw[->,>=stealth]
    (A) edge (C)
    (B) edge (D) ;
\draw[->,>=stealth]
    (C) -- (D) node[above, midway]{$\quad\cong$};
\end{tikzpicture}\]
where the local analytic isomorphism on the left is given by a particular ratio of the functions \eqref{osci}.
\end{itemize}
The fact that these two conditions are compatible follows from the relations \eqref{airy} and \eqref{aaaa}, after recalling that the M{\"o}bius transformation $T$ is induced by the action of $\Sigma$ on the space
$\P\Hom_{\Z}(K_0(\D_\infty),\bC)$.

To construct the required map $h_\infty$, we first consider the subset $P_\infty\subset \bC$ which is the union of the lower and upper half-planes $\pm \Im(t)>0$ in the coarse moduli space, glued along the segment
$(0,\infty)=(\mu_3,\infty)$. We can glue  the isomorphism $f_\infty\colon \H\to R_\infty^-$   of  Section \ref{conf} to its complex conjugate  to obtain an isomorphism $f_\infty\colon P_\infty\to R_\infty$. Composing
with the inverse of the isomorphism $g_\infty\colon \cU_\infty\to R_\infty$ of Prop. \ref{rain} gives a biholomorphic map
$h_\infty\colon P_\infty\to \cU_\infty$. We can again lift $P_\infty$ into the universal cover, via the map $a=t^{1/3}$, and so  view $P_\infty$ as the sector in the $a$--plane bounded by the rays of phase $\pm \pi/3$. The resulting map  satisfies condition (ii) by construction.
We then extend the map $h_\infty$ to a dense open subset of $\bC$ using condition (i).  The compatibility of the two conditions ensures that this extension also satisfies condition (ii). Finally, we can  extend $h_\infty$  over the missing rays using the vertical local isomorphisms in the above commuting diagram.
\end{proof}

\subsection{}The final step is to lift Theorem \ref{proj} to obtain a proof of our main result Theorem \ref{first}. Consider first the case $3\leq n<\infty$.
We have a diagram of complex manifolds and holomorphic maps
\[\begin{tikzpicture}
\draw(0,1) node (A) {$\widetilde{\C^2\setminus \Delta}$}
     (5,1) node (B) {$\Stab_*(\D_n)$}
     (0,-1) node (C) {$\widetilde{\P(2,3)\setminus \Delta}$}
     (5,-1) node (D) {$\P\Stab_*(\D_n)$};
\draw[->,>=stealth]
    (A) edge (C)
    (B) edge (D)
    (C) -- (D) node[above, midway]{$h_n$};
\end{tikzpicture}\]
The vertical arrows are $\C$-bundles, and the bottom horizontal arrow $h_n$ is the isomorphism constructed in the proof of Theorem \ref{proj}.
We would like to complete the diagram by filling in an upper horizontal isomorphism satisfying the property claimed in Theorem \ref{first}. Note that by construction the central charge map
\[\P \Stab(\D_n) \to \P^1\]
corresponds under the isomorphism $h_n$ to the  map given by ratios of the functions $\phi^{(i)}_n(a,b)$ of Prop. \ref{endy}. These functions are well-defined on the universal cover of $\C^2\setminus\Delta$,  and there is therefore a unique way to fill in the upper arrow to give an isomorphism
\begin{equation}
\label{col} h_n\colon \widetilde{\C^2\setminus \Delta} \to\Stab_*(\D_n)\end{equation} so that the composition with the  central charge map on $\Stab_*(\D_n)$ is given by  the twisted periods $\phi^{(i)}_n(a,b)$.

The isomorphism \eqref{row} lifts to an isomorphism
\begin{equation}
\label{row2}\rho\colon \pi_1(\C^2\setminus\Delta)\to \Sph(\D_n),\end{equation}
by mapping the extra generator corresponding to a loop in the fibre of the $\C^*$-bundle  \[\pi\colon \C^2\setminus \Delta \to \P(2,3)\setminus\Delta\] to the element $[3n-4]$. This loop is given explicitly by a path of the form \[(a(\theta),b(\theta))=(e^{4\pi i \theta} \cdot a, e^{6\pi i \theta} \cdot b), \qquad \theta\in[0,1].\] By  the formula \eqref{eq:twisted p} this means that the cebtral charges of the objects $S_i$ vary as
\[Z(S_i)(\theta)=\phi^{(i)}(e^{4\pi i \theta} \cdot a, e^{6\pi i \theta} \cdot b)= e^{(3n-4)\pi i \theta}\cdot \phi^{(i)}(a,b)= e^{(3n-4)\pi i \theta}\cdot Z(S_i),\]
and hence the  phases of these objects increase by $3n-4$. This calculation shows  that the map \eqref{col}  we constructed  is equivariant with respect to the group isomorphism \eqref{row2}, and passing to quotients then gives the statement of Theorem \ref{first} (a).

In the case $n=\infty$ we have a similar diagram
\[\begin{tikzpicture}
\draw(0,1) node (A) {$\C^2$}
     (5,1) node (B) {$\Stab_*(\D_\infty)$}
     (0,-1) node (C) {$\C$}
     (5,-1) node (D) {$\P\Stab_*(\D_\infty)$};
\draw[->,>=stealth]
    (A) edge (C)
    (B) edge (D)
    (C) -- (D) node[above, midway]{$h_\infty$};
\end{tikzpicture}\]
in which the vertical arrows are again $\C$-bundles.
The bundle on the left is just the projection $\C^2\to \C$ given by $(a,b)\mapsto a$. By construction, the central charge map
\[\P \Stab_*(\D_\infty) \to \P^1\]
is given by ratios of the functions  $\phi^{(i)}_\infty(a,b)$ of Prop. \ref{osc}. These functions lift to $\C^2$ and there is  therefore a unique way to fill in the upper arrow with an isomorphism
\begin{equation}
\label{finishedatlast}h_\infty\colon \C^2\to \Stab_*(\D_\infty)\end{equation}
so that the  composition with the central charge map on $\Stab(\D_\infty)$ is given by   the  oscillating integrals $\phi^{(i)}_\infty(a,b)$. This completes the proof of Theorem \ref{first}.

Recall that in this case there is an isomorphism of groups
\[\Z\to \Aut(\D_\infty), \qquad 1\mapsto \Sigma.\]
 The map \eqref{finishedatlast} can be made equivariant  by letting $\Z$ act on $\bC^2$ via  \[(a,b)\mapsto (e^{2\pi i/3} \cdot a, b+\pi i/3).\] The element $3\mapsto \Sigma^3=[1]$ then fixes $a$ and acts by $b\mapsto b+\pi i$.


%
%

\end{document}